\documentclass{siamltex}
\usepackage{amsmath,amssymb}
\usepackage{amsfonts}
\usepackage{graphicx}
\usepackage{color}
\usepackage[]{algorithm2e}
\usepackage{algorithmic}
\usepackage{tikz}
\usepackage{multirow}
\usepackage{multicol}
\usepackage{paralist}
\usepackage{pgfplots}
\usepackage{color}
\usepackage{qtree}
\usepackage{caption}
\usepackage{subcaption}
\pgfplotsset{compat=newest}
\pgfplotsset{plot coordinates/math parser=false}

\usepgfplotslibrary{external} 

\newcommand{\transfertensorCol}[1]
{
\begin{tikzpicture}[scale=1.2]
\pgfmathsetmacro{\cubex}{0.4}
\pgfmathsetmacro{\cubey}{0.4}
\pgfmathsetmacro{\cubez}{0.4}
\draw[black,fill=#1] (0,0,0) -- ++(-\cubex,0,0) -- ++(0,-\cubey,0) -- ++(\cubex,0,0) -- cycle;
\draw[black,fill=#1] (0,0,0) -- ++(0,0,-\cubez) -- ++(0,-\cubey,0) -- ++(0,0,\cubez) -- cycle;
\draw[black,fill=#1] (0,0,0) -- ++(-\cubex,0,0) -- ++(0,0,-\cubez) -- ++(\cubex,0,0) -- cycle;
\node at (0.7,-0.1) {$^{k\times k\times k}$};
\node at (-1.3,-0.1) {};
\end{tikzpicture}
}

\newcommand{\transfertensorAdd}[2]
{
\begin{tikzpicture}[scale=1.2]
\pgfmathsetmacro{\cubex}{0.4}
\pgfmathsetmacro{\cubey}{0.4}
\pgfmathsetmacro{\cubez}{0.4}
\pgfmathsetmacro{\bigcubex}{2*\cubex}
\pgfmathsetmacro{\bigcubey}{2*\cubey}
\pgfmathsetmacro{\bigcubez}{2*\cubez}
\draw[black] (\cubex,-\bigcubey,-\bigcubez) -- (-\cubex,-\bigcubey,-\bigcubez) -- (-\cubex,-\bigcubey,0);
\draw[black] (-\cubex,-\bigcubey,-\bigcubez) -- (-\cubex,0,-\bigcubez) -- (\cubex,0,-\bigcubez) -- (\cubex,-\cubey,-\bigcubez);
\draw[black,fill=#1] (0,0,0) -- ++(-\cubex,0,0) -- ++(0,-\cubey,0) -- ++(\cubex,0,0) -- cycle;
\draw[black,fill=#1] (0,0,0) -- ++(0,0,-\cubez) -- ++(0,-\cubey,0) -- ++(0,0,\cubez) -- cycle;
\draw[black,fill=#1] (0,0,0) -- ++(-\cubex,0,0) -- ++(0,0,-\cubez) -- ++(\cubex,0,0) -- cycle;
\draw[black,fill=#2] (\cubex,-\cubey,-\cubez) -- ++(-\cubex,0,0) -- ++(0,-\cubey,0) -- ++(\cubex,0,0) -- cycle;
\draw[black,fill=#2] (\cubex,-\cubey,-\cubez) -- ++(0,0,-\cubez) -- ++(0,-\cubey,0) -- ++(0,0,\cubez) -- cycle;
\draw[black,fill=#2] (\cubex,-\cubey,-\cubez) -- ++(-\cubex,0,0) -- ++(0,0,-\cubez) -- ++(\cubex,0,0) -- cycle;
\draw[black] (0,0,0) -- (\cubex,0,0) -- (\cubex,-\bigcubey,0) -- (-\cubex,-\bigcubey,0) -- (-\cubex,-\cubey,0);
\draw[black] (\cubex,-\bigcubey,0) -- (\cubex,-\bigcubey,-\cubez);
\draw[black] (\cubex,0,0) -- (\cubex,0,-\bigcubez);
\draw[black] (-\cubex,0,-\cubez) -- (-\cubex,0,-\bigcubez);
\node at (1.4,-0.2) {$^{2k\times 2k\times 2k}$};
\node at (-2,-0.1) {};
\end{tikzpicture}
}

\newcommand{\roottransfertensor}
{
\begin{tikzpicture}[scale=1.2]
\pgfmathsetmacro{\cubex}{0.4}
\pgfmathsetmacro{\cubey}{0.4}
\draw[black] (0,0,0) -- ++(-\cubex,0,0) -- ++(0,-\cubey,0) -- ++(\cubex,0,0) -- cycle;
\node at (-0.55,-0.2) {$^{k}$};
\node at (-0.2,0.1) {$^{k}$};
\node at (0.3,-0.2) {};
\end{tikzpicture}
}

\newcommand{\roottransfertensorCol}[1]
{
\begin{tikzpicture}[scale=1.2]
\pgfmathsetmacro{\cubex}{0.4}
\pgfmathsetmacro{\cubey}{0.4}
\draw[black,fill=#1] (0,0,0) -- ++(-\cubex,0,0) -- ++(0,-\cubey,0) -- ++(\cubex,0,0) -- cycle;
\node at (-0.55,-0.2) {$^{k}$};
\node at (-0.2,0.1) {$^{k}$};
\node at (0.3,-0.2) {};
\end{tikzpicture}
}

\newcommand{\roottransfertensorAdd}[2]
{
\begin{tikzpicture}[scale=1.2]
\pgfmathsetmacro{\cubex}{0.4}
\pgfmathsetmacro{\cubey}{0.4}
\draw[black,fill=#1] (0,0,0) -- ++(-\cubex,0,0) -- ++(0,-\cubey,0) -- ++(\cubex,0,0) -- cycle;
\draw[black,fill=#2] (\cubex,-\cubey,0) -- ++(-\cubex,0,0) -- ++(0,-\cubey,0) -- ++(\cubex,0,0) -- cycle;
\draw[black] (\cubex,0,0) -- ++(-\cubex,0,0) -- ++(0,-\cubey,0) -- ++(\cubex,0,0) -- cycle;
\draw[black] (0,-\cubey,0) -- ++(-\cubex,0,0) -- ++(0,-\cubey,0) -- ++(\cubex,0,0) -- cycle;
\node at (-0.6,-0.4) {$^{2k}$};
\node at (0,0.1) {$^{2k}$};
\node at (0.3,-0.2) {};
\end{tikzpicture}
}

\newcommand{\leafframeCol}[1]
{
\begin{tikzpicture}[scale=1.2]
\pgfmathsetmacro{\cubex}{0.4}
\pgfmathsetmacro{\cubey}{0.9}
\draw[black,fill=#1] (0,0,0) -- ++(-\cubex,0,0) -- ++(0,-\cubey,0) -- ++(\cubex,0,0) -- cycle;
\node at (-0.2,-1.1) {$^{n\times k}$};
\end{tikzpicture}
}

\newcommand{\leafframeAdd}[2]
{
\begin{tikzpicture}[scale=1.2]
\pgfmathsetmacro{\cubex}{0.4}
\pgfmathsetmacro{\cubey}{0.9}
\draw[black,fill=#1] (0,0,0) -- ++(-\cubex,0,0) -- ++(0,-\cubey,0) -- ++(\cubex,0,0) -- cycle;
\draw[black,fill=#2] (\cubex,0,0) -- ++(-\cubex,0,0) -- ++(0,-\cubey,0) -- ++(\cubex,0,0) -- cycle;
\node at (0,-1.2) {$^{n\times 2k}$};
\end{tikzpicture}
}

\newcommand{\leafframeindexCol}[3]
{
\begin{tikzpicture}[scale=1.2]
\pgfmathsetmacro{\cubex}{0.4}
\pgfmathsetmacro{\cubey}{0.9}
\draw[black,fill=#1] (0,0,0) -- ++(-\cubex,0,0) -- ++(0,-\cubey,0) -- ++(\cubex,0,0) -- cycle;
\node at (-0.2,-1.1) {$^{n\times k}$};
\draw[fill=#2] (-0.4,-0.1*#3 -0.1) rectangle (0,-0.1*#3);
\end{tikzpicture}
}

\newcommand{\transfertensormultresultCol}[4]
{
\begin{tikzpicture}[scale=1.2]
\pgfmathsetmacro{\cubex}{0.4}
\pgfmathsetmacro{\cubey}{0.4}
\pgfmathsetmacro{\cubez}{0.4}
\draw[black,fill=#1] (0,0,0) -- ++(-\cubex,0,0) -- ++(0,-\cubey,0) -- ++(\cubex,0,0) -- cycle;
\draw[black,fill=#1] (0,0,0) -- ++(0,0,-\cubez) -- ++(0,-\cubey,0) -- ++(0,0,\cubez) -- cycle;
\draw[black,fill=#1] (0,0,0) -- ++(-\cubex,0,0) -- ++(0,0,-\cubez) -- ++(\cubex,0,0) -- cycle;
\draw[fill=#3] (0.25,-0.25) rectangle (0.35,0.15);
\draw[fill=#2] (-0.85,-0.2) rectangle (-0.45,-0.1);
\node at (0.5,-0.1) {$=$};
\pgfmathsetmacro{\smallcubex}{0.1}
\pgfmathsetmacro{\smallcubey}{0.1}
\draw[black,fill=#4] (0.77,-0.1,0.1) -- ++(-\smallcubex,0,0) -- ++(0,-\smallcubey,0) -- ++(\smallcubex,0,0) -- cycle;
\draw[black,fill=#4] (0.77,-0.1,0.1) -- ++(0,0,-\cubez) -- ++(0,-\smallcubey,0) -- ++(0,0,\cubez) -- cycle;
\draw[black,fill=#4] (0.77,-0.1,0.1) -- ++(-\smallcubex,0,0) -- ++(0,0,-\cubez) -- ++(\smallcubex,0,0) -- cycle;
\end{tikzpicture}
}

\newcommand{\roottransfertensormultCol}[4]
{
\begin{tikzpicture}[scale=1.2]
\pgfmathsetmacro{\cubex}{0.4}
\pgfmathsetmacro{\cubey}{0.4}
\draw[black,fill=#1] (0,0,0) -- ++(-\cubex,0,0) -- ++(0,-\cubey,0) -- ++(\cubex,0,0) -- cycle;
\draw[fill=#3] (0.1,-0.4) rectangle (0.2,0);
\draw[fill=#2] (-0.85,-0.2) rectangle (-0.45,-0.1);
\node at (1.3,-0.2) {$=$\,#4};
\end{tikzpicture}
}

\newtheorem{remark}[theorem]{Remark}

 \evensidemargin=\oddsidemargin
 \newdimen\dummy
 \dummy=\oddsidemargin
\newcommand{\range}{\mathop{\rm range}}

\newcommand{\mytimes}{\mathop{\mbox{\Large$\times$}}}

\DeclareMathOperator{\cond}{cond}

\begin{document}
\title{Distributed Hierarchical SVD in the Hierarchical Tucker Format}
\author{
  Lars Grasedyck\thanks{Institut f\"ur Geometrie und Praktische Mathematik, 
    RWTH Aachen, Templergraben 55, 52056 Aachen, Germany. 
    Email: {\tt \{lgr,loebbert\}@igpm.rwth-aachen.de}. The authors gratefully acknowledge support by the
    DFG priority programme 1648 (SPPEXA) under grant GR-3179/4-2.
  }
  \and
  Christian L\"obbert$^*$
}

\maketitle
\sloppy
\begin{abstract}
  We consider tensors in the Hierarchical Tucker format and suppose the tensor data to be distributed 
  among several compute nodes. We assume the compute nodes to be in a one-to-one correspondence with the
  nodes of the Hierarchical Tucker format such that connected nodes can communicate with each other.
  An appropriate tree structure in the Hierarchical Tucker format then allows for the parallelization of basic arithmetic
  operations between tensors
  with a parallel runtime which grows like $\log(d)$, where $d$ is the tensor dimension.
  We introduce parallel algorithms for several tensor operations, some of which can be applied to 
  solve linear equations $\mathcal{A}X=B$ directly in the Hierarchical Tucker format 
  using iterative methods like
  conjugate gradients or multigrid. We present weak scaling studies, which provide evidence that the runtime
  of our algorithms indeed grows like $\log(d)$. Furthermore, we present 
  numerical experiments in which we apply our algorithms to solve a parameter-dependent diffusion equation
  in the Hierarchical Tucker format by means of a multigrid algorithm.

  Keywords: Hierarchical Tucker, HT, Multigrid, Parallel algorithms, SVD, Tensor arithmetic

\end{abstract}

\section{Introduction}\label{Sec:intro}
High dimensional tensors may e.g. arise in the context of parameter-dependent problems. 
Consider a linear equation
\begin{equation*}
A(p_1,\, p_2, \ldots, \, p_d)\cdot \vec{x}(p_1,\,p_2,\ldots ,\, p_d ) = \vec{b}(p_1,\, p_2, \,\ldots ,\, p_d),
\end{equation*}
where the matrix $A$, the right-hand side $\vec{b}$, and with that also the solution $\vec{x}$,
depend on a (possibly high) number of parameters.
Then the solution vector $\vec{x}$ can be regarded as a tensor $X$ (i.e. a multidimensional array) of dimension 
$d+1$, where we have one tensor dimension for each of the $d$ parameters plus the vector dimension of $\vec{x}$:
\begin{equation*}
x_i(p_1,\,p_2,\,\ldots ,\, p_d)\quad = \quad X( p_1,\, p_2,\,\ldots ,\,p_d,\, i).
\end{equation*}
In the same way, the right-hand side $\vec{b}$ and the matrix $A$ can be regarded as $(d+1)$-dimensional 
tensors $B$ and $\mathcal{A}$, where the two vector indices $i, j$ of the matrix $A$ are combined to one index
$(i,j)$:
\begin{equation}\label{equation:operatorTensor}
A_{ij}(p_1,\, p_2,\,\ldots ,\, p_d )\quad = \quad \mathcal{A}(p_1,\,p_2,\, \ldots ,\, p_d,\,(i,j)).
\end{equation}
If low-rank representations/approximations of the tensors $\mathcal{A}$ and $B$ are available 
and if basic arithmetic operations
can directly be performed in the underlying low-rank format, then an approximate solution $X$ 
of the tensor equation $\mathcal{A}X=B$ can be
obtained by applying some iterative method directly in the low-rank format.
The solution tensor $X$ will then contain all solutions $\vec{x}(p_1,p_2,\ldots ,p_d)$ for any
parameter combination and the solutions for single parameter combinations can easily be extracted from $X$.
Furthermore several postprocessing operations like computing the mean over (all) parameter combinations 
or computing expected
values with respect to a probability distribution of the parameters can directly be carried out for the solution $X$
in the low-rank format.

In this article we present parallel algorithms for basic arithmetic operations on tensors, where we
choose the Hierarchical Tucker format as low-rank format for tensors.
The Hierarchical Tucker format is briefly introduced in 
Section~\ref{section:distributedtensors}. 
For a more detailed description of low-rank tensor approximation techniques 
we refer to the
literature survey \cite{GrKrTo13}.
For our algorithms the tensors may be stored distributed over
several compute nodes. As a consequence the algorithms can also be used for postprocessing operations on 
tensors which stem from a parallel sampling method on distinct compute nodes, as e.g. described in 
\cite{GrKriLoe15}.

In Section~\ref{section:arithmetic} we give an overview of our parallel algorithms for 
the Hierarchical Tucker format. The Hierarchical
Tucker format is based on a binary tree (cf. Section~\ref{section:distributedtensors}, 
Fig.~\ref{figure:hierarchy}). For this article we assume the data of each tree node to be
stored on its own compute node. Then our algorithms typically require communication between
compute nodes which are neighbors with respect to the binary tree. 
Assuming the number of dimensions $d$ to be a power of two, $d = 2^p$, $p\in\mathbb{N}$, and
choosing the tree $T_d$ such that the number of tree levels is minimized, our algorithms can
run in parallel on all compute nodes of the same tree level. 
In Section~\ref{section:arithmetic} we refer to this as \textit{level-wise parallelization}.
Since we have $p+1 = \log_2(d) + 1$ tree levels, we expect the parallel runtime of our 
algorithms to grow like $\log(d)$.

In Section~\ref{section:runtime} we present runtime tests which provide evidence that the parallel
runtime of our algorithms indeed grows like $\log(d)$ for arithmetic operations between
tensors of dimension $d$.

In Section~\ref{Sec:multigrid} we apply our algorithms to solve a parameter-dependent toy
problem in the Hierarchical Tucker format by means of iterative methods. We chose a diffusion
equation with piecewise constant diffusion coefficients which are controlled by 9 parameters.
Applying a Finite Element discretization we obtain a linear system, the matrix of which depends
on the 9 parameters. Due to its affine structure with respect to the parameters, 
the system matrix can directly be
represented as linear operator in the Hierarchical Tucker format, i.e. we can use iterative
methods to solve the linear equation, 
where all arithmetic operations are carried out
between tensors in the Hierarchical Tucker format.
We illustrate this for a multigrid method containing a Richardson method as smoother.
On the coarsest grid level we use a CG method to solve the defect equation. This numerical
experiment shows less the benefit of the parallelization 
(which is demonstrated in Section~\ref{section:runtime}) but rather confirms that our algorithms
can in principle be used together with iterative methods.

\section{Distributed tensors in the Hierarchical Tucker format}\label{section:distributedtensors}
We consider real-valued tensors $A\in\mathbb{R}^{\mathcal{I}}$ for some product index set 
$\mathcal{I} = \mathcal{I}_1\times\cdots\times\mathcal{I}_d$, where the $\mathcal{I}_{\mu}$, $\mu = 1,\ldots ,d$,
are index sets of size $n_{\mu}\in\mathbb{N}$, e.g. $\mathcal{I}_{\mu} = \{1,\ldots ,n_{\mu}\}$.
The number $d\in\mathbb{N}$ is referred to as \textit{dimension} of the tensor.
A real-valued tensor is thus a mapping $A\,:\, \mathcal{I}\to\mathbb{R}$ from a product index set
$\mathcal{I}$ to the set $\mathbb{R}$ of real numbers.

Assuming the index sets $\mathcal{I}_{\mu}$ to be all of the same size $n:=n_{\mu}$, $\mu = 1,\ldots ,d$, a tensor 
$A\in\mathbb{R}^{\mathcal{I}}$ consists of $n^d$ entries and may be regarded as a vector in $\mathbb{R}^{n^d}\cong \mathbb{R}^{\mathcal{I}}$. To emphasize the tensor structure, we prefer the notation
$A\in\mathbb{R}^{\mathcal{I}}$, $\mathcal{I} = \mathcal{I}_1\times\cdots\times \mathcal{I}_d$,
instead of $A\in\mathbb{R}^{n^d}$.

Even for moderate numbers $n$ and $d$ the exponential 
growth of the number $n^d$ of tensor entries makes it impossible to store the tensor entry-by-entry, not to mention the costs for arithmetic operations.
This \textit{curse of dimensionality} led to a variety of different tensor formats which make use of low rank structures
in the tensors.

In this article the \textit{Hierarchical Tucker format} \cite{GrKrTo13,Gr10,Ha12} is used, which is based on a hierarchy of the
tensor dimensions $\{1,\ldots ,d\}$, as depicted in Fig~\ref{figure:hierarchy} for $d=8$.
\begin{figure}
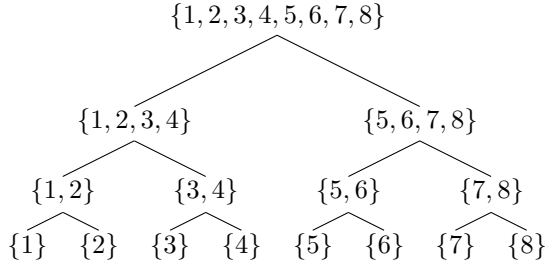

\Tree 
[.$\{1,2,3,4,5,6,7,8\}$
    [.$\{1,2,3,4\}$
        [.$\{1,2\}$
            $\{1\}$ $\{2\}$
        ].$\{1,2\}$
        [.$\{3,4\}$
            $\{3\}$ $\{4\}$
        ].$\{3,4\}$
    ].$\{1,2,3,4\}$
    [.$\{5,6,7,8\}$
        [.$\{5,6\}$
            $\{5\}$ $\{6\}$
        ].$\{5,6\}$
        [.$\{7,8\}$
            $\{7\}$ $\{8\}$
        ].$\{7,8\}$
    ].$\{5,6,7,8\}$
].$\{1,2,3,4,5,6,7,8\}$
\caption{Balanced dimension tree for dimension $d=8$.}
\label{figure:hierarchy}
\end{figure}
The \textit{dimension tree} of Fig.~\ref{figure:hierarchy} is denoted by $T_{8}$, since $d=8$. Of course one could think of
many different ways how to split a node $\{\mu,\ldots ,\nu\}$ of the tree $T_{d}$ into its two sons. However
we will always assume a \textit{balanced} decomposition where $\{\mu,\ldots ,\nu\}$ is subdivided into the two sons
\begin{equation}\label{equation:balanced}
\biggl\{ \mu\,,\; \ldots \,,\; \biggl\lfloor \frac{\mu + \nu -1}{2} \biggr\rfloor \biggr\}\qquad\mbox{and}\qquad 
\biggl\{ \biggl\lfloor \frac{\mu + \nu -1}{2} \biggr\rfloor + 1\, ,\;\ldots \,,\;\nu \biggr\}.
\end{equation}
This \textit{balanced tree} minimizes the number of tree levels: $\mathrm{depth}(T_{d}) = \lceil\log_2(d)\rceil$, whereas a subdivision
of $\{\mu,\ldots ,\nu\}$ into $\{\mu\}$ and $\{\mu +1,\ldots ,\nu\}$ would result in $\mathrm{depth}(T_{d}) = d-1$.
Since our algorithms will allow for level-wise parallelization, the balanced tree appears to be the best choice here.

For simplicity we assume the dimension $d$ to be a power of two: $d = 2^{p}$ for some $p\in\mathbb{N}$, 
i.e. the leaves $\{\mu\}$, 
$\mu = 1,\ldots ,d$, are all on the highest tree level and we always have $2^{\ell}$ tree nodes on level $\ell$, where
level 0 is the root and level $p$ is the highest level.

\subsection{The Hierarchical Tucker format}\label{subsection:htuckerformat}
Consider a tensor $A\in\mathbb{R}^{\mathcal{I}}$ for a product index set $\mathcal{I}$ of dimension $d$, as above. 
Then for a subset
$t\subset \{1,\ldots ,d\}$ let $[t] = \{1,\ldots ,d\}\setminus t$ denote the complement of $t$
and we define the product index sets $\mathcal{I}_t$ and $\mathcal{I}_{[t]}$ by
\begin{equation*}
\mathcal{I}_t = \mytimes_{\mu\in t}\mathcal{I}_{\mu}\qquad\mbox{and}\qquad 
\mathcal{I}_{[t]} = \mytimes_{\mu\in [t]}\mathcal{I}_{\mu}.
\end{equation*}
\begin{definition}[Matricization]\label{definition:matricization}
The matricization $\mathcal{M}_t(A)\in\mathbb{R}^{\mathcal{I}_t\times \mathcal{I}_{[t]}}$ 
of the tensor $A\in\mathbb{R}^{\mathcal{I}}$ with respect to a subset $t\subset\{1,\ldots ,d\}$ is defined via
\begin{equation*}
\mathcal{M}_t(A)_{(\mathbf{i}_{t}\, ,\, \mathbf{i}_{[t]})} = A_{(i_1,\ldots ,i_d)}
\qquad\mbox{for}
\qquad (i_1,\ldots ,i_d)\in \mathcal{I} = \mathcal{I}_1\times\cdots\times \mathcal{I}_d,
\end{equation*}
where we used the short notation $\mathbf{i}_{s} = (i_{\mu})_{\mu\in s}\in\mathcal{I}_s$ for subsets $s\subset\{1,\ldots ,d\}$.
\end{definition}
For a tensor $A\in\mathbb{R}^{\mathcal{I}}$ the \textit{Hierarchical Tucker format} is based on matricizations of $A$
with respect to the subsets $t\subset\{1,\ldots ,d\}$ of some underlying dimension tree $T_{d}$.

For each node $t\in T_{d}$ of the tree let $U_t$ be a matrix, the columns of which
span the same vector space as the columns of the matricization $\mathcal{M}_t(A)$:
\begin{equation*}
\range(U_t) = \range(\mathcal{M}_t(A)).
\end{equation*}
Such a matrix will be called \textit{frame} from now on.

If $\rank(\mathcal{M}_t(A)) = k_t$, 
one can find a frame $U_t$ with $k_t$ columns, i.e. $U_t\in\mathbb{R}^{\mathcal{I}_t\times k_t}$.
For the root $D := \{1,\ldots ,d\}$ of $T_{d}$ the matricization 
$\mathcal{M}_D(A)\in\mathbb{R}^{\mathcal{I}\times 1}$ is just the 
rearrangement of the tensor $A$ in one long column and one can choose $U_D = \mathcal{M}_D(A)$,
i.e. the frame $U_D$ becomes as hard to handle as the tensor $A$ itself. 

In the Hierarchical Tucker format the frames $U_t$ are only stored for the leaves $t = \{\mu\}$ 
of the tree, for which the $U_{\{\mu\}}$ are of size $n_{\{\mu\}}\times k_{\{\mu\}}$, $\mu = 1,\ldots ,d$. For a non-leaf node $t$ the row index set $\mathcal{I}_t$ of $U_t$ could already be too
large, which is why we do not store $U_t$ explicitly for non-leaf nodes.

Each non-leaf node $t$ has two sons $s_1$ and $s_2$ fulfilling $t = s_1\dot{\cup} s_2$, 
and one can find coefficients
$(B_t)_{i,j,\ell}$, such that
\begin{equation}\label{equation:transfer}
(U_t)_{-,i} = \sum_{j = 1}^{k_{s_1}}\sum_{\ell = 1}^{k_{s_2}}(B_t)_{i,j,\ell}\bigl((U_{s_1})_{-,j}\otimes (U_{s_2})_{-,\ell}\bigr),
\end{equation}
where $(U_t)_{-,i}\in\mathbb{R}^{\mathcal{I}_t}$ denotes the $i$-th column of $U_t$.

Instead of $U_t\in\mathbb{R}^{\mathcal{I}_t\times k_t}$ only the coefficients
$(B_t)_{i,j,\ell}$ are stored, where ${i\in\{1,\ldots ,k_t\}}$, ${j\in\{1,\ldots , k_{s_1}\}}$ and 
${\ell\in\{1,\ldots ,k_{s_2}\}}$. This results in a tensor
$B_t$ of size $k_t\times k_{s_1}\times k_{s_2}$ for each non-leaf node, referred to as \textit{transfer tensor}. 
Note that we have $k_D = 1$ for the root
($U_D$ is one single column), i.e. $B_D$ is a $k_{t_1}\times k_{t_2}$-matrix, where
$\mathrm{sons}(D) = \{t_1,t_2\}$.

For any index $(i_1,\ldots ,i_d)\in\mathcal{I}$ we have 
\begin{equation*}
A_{(i_1,\ldots ,i_d)} = \mathcal{M}_D(A)_{(i_1,\ldots ,i_d)\, ,\, 1} = (U_D)_{(i_1,\ldots ,i_d)\, ,\, 1}
\end{equation*}
and the corresponding tensor entry can be evaluated by recursively using the relation~\eqref{equation:transfer}.

A Hierarchical Tucker tensor (for short, $\mathcal{H}$-tensor) can be depicted by 
the underlying tree $T_{d}$, 
cf. Fig~\ref{figure:htensor}.
\begin{figure}
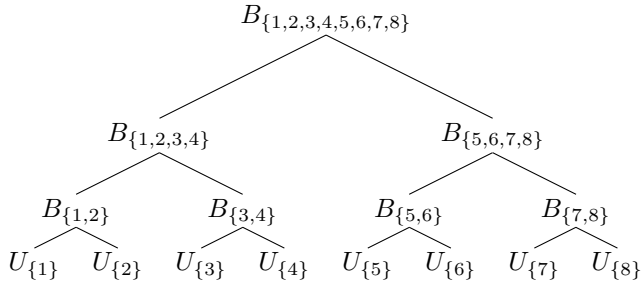

\Tree 
[.$B_{\{1,2,3,4,5,6,7,8\}}$
    [.$B_{\{1,2,3,4\}}$
        [.$B_{\{1,2\}}$
            $U_{\{1\}}$ $U_{\{2\}}$
        ].$B_{\{1,2\}}$
        [.$B_{\{3,4\}}$
            $U_{\{3\}}$ $U_{\{4\}}$
        ].$B_{\{3,4\}}$
    ].$B_{\{1,2,3,4\}}$
    [.$B_{\{5,6,7,8\}}$
        [.$B_{\{5,6\}}$
            $U_{\{5\}}$ $U_{\{6\}}$
        ].$B_{\{5,6\}}$
        [.$B_{\{7,8\}}$
            $U_{\{7\}}$ $U_{\{8\}}$
        ].$B_{\{7,8\}}$
    ].$B_{\{5,6,7,8\}}$
].$B_{\{1,2,3,4,5,6,7,8\}}$
\caption{Hieararchical tensor of dimension $d=8$ with underlying tree $T_{d}$ of Fig.~\ref{figure:hierarchy}}
\label{figure:htensor}
\end{figure}
Writing $n:=\max\{n_{\mu}\mid \mu=1,\ldots ,d\}$ and $k:=\max\{k_t\mid t\in T_{d}\}$, 
we can estimate 
the number of entries stored in the Hierarchical Tucker format by
\begin{equation}\label{equation:storageComplexity}
\#\mathrm{leaves}\cdot n\cdot k + \#(\mathrm{inner\, nodes})\cdot k^3 + k^2 = dnk + (d-2)k^3 + k^2 = \mathcal{O}(d).
\end{equation}
The numbers $k_t$, $t\in T_d$, are called \textit{Hierarchical ranks} of the tensor. To be precise, one can 
distinguish between the actual \textit{Hierarchical representation ranks} $(k_t)_{t\in T_d}$ and the minimal
ranks, since the $k_t$, $t\in T_d$, need not necessarily to be minimal: In the case
\begin{equation*}
k_t > \rank(\mathcal{M}_t (A))
\end{equation*}
we can reduce the complexity of the representation until we have $k_t = \rank(\mathcal{M}_t (A))$.
This will, of course, always be the goal, we will even truncate tensors down to lower Hierarchical ranks,
which means to find an approximating tensor of lower Hierarchical ranks.

Since we only deal with tensors in the Hierarchical Tucker format, we will sometimes use the short name
\textit{tensor ranks} or just \textit{ranks}.

In our runtime tests, we will, for simplicity, choose $k_t = k$ for all $t\in T_d\setminus \{1,\ldots ,d\}$, 
and therefore speak of \textit{the (Hierarchical) rank}.

\subsection{Distributing $\mathcal{H}$-tensors}\label{subsection:distribution}
In this article we assume the tensor $A\in\mathbb{R}^{\mathcal{I}}$ to be stored in the Hieararchical Tucker format,
where the underlying dimension tree $T_{d}$ is of balanced type according to \eqref{equation:balanced}.
Furthermore, we assume the data to be distributed among several compute nodes, where the data of each tree node is
stored on its own compute node. For tensor dimension $d$ this leads to $2d-1$ compute nodes. The tensor in Fig.~\ref{figure:htensor} would be distributed among 15 compute nodes. 
Communication is supposed to be possible via the edges of the tree, i.e. each
node can communicate with its sons and its father.

The distribution of the tensor data over the $2d-1$ compute nodes was realized by \texttt{MPI}.
In our implementation we store the structure of the tree $T_{d}$ on each \texttt{MPI}~process, 
which includes
the \texttt{MPI}~ranks (i.e. the IDs of the \texttt{MPI}~processes) 
for all the other nodes. In fact it would be sufficient that each \texttt{MPI}~process knows the 
\texttt{MPI}~ranks of its sons and its father.

As mentioned in Section~\ref{Sec:intro}, our algorithms will typically run in parallel
on \texttt{MPI} processes which belong to the same tree level in $T_d$ (level-wise parallelization).

Distributed $\mathcal{H}$-tensors may arrise from parallel tensor sampling \cite{GrKriLoe15}, where a given output
function for solutions of a parameter dependent problem is approximated in the Hierarchical Tucker format.

If $2d-1$ \texttt{MPI} processes are not available, the distribution of the tree nodes to 
the available \texttt{MPI}
processes would be an interesting issue. This question will, however, not be discussed in this
article and may be part of further investigations.

\section{Parallel arithmetic for $\mathcal{H}$-tensors}\label{section:arithmetic}
Given a tensor that has either been assembled in parallel on distributed compute nodes or has been distributed
due to its extent of storage, one would like to run tensor computations in parallel to the highest possible extent.
In this section we demonstrate how arithmetic operations in the Hierarchical Tucker format 
can be parallelized level-wise with respect to the underlying tree
$T_{d}$, where we assume the tree nodes to correspond to the compute nodes (cf.
Section~\ref{subsection:distribution}). By using a balanced tree $T_{d}$ we can thus expect 
the parallel computing time to grow like $\mathrm{depth}(T_{d}) =\log_2(d)$, assuming  load
balancing between the nodes.

In this section we consider the inner product $\langle A,B\rangle\in\mathbb{R}$,
\begin{equation}\label{equation:innerproduct}
\langle A, B\rangle = \sum_{(i_1,\ldots ,i_d)\in\mathcal{I}}A_{(i_1,\ldots ,i_d)}\cdot B_{(i_1,\ldots ,i_d)}\quad ,
\end{equation}
of two $\mathcal{H}$-tensors $A,B\in\mathbb{R}^{\mathcal{I}}$, 
the addition $A+B$, and the application of an operator $\mathcal{L}$ to an $\mathcal{H}$-tensor $A$,
both followed by the truncation of the result, either down to prescibed ranks $(k_t)_{t\in T_{d}}$ or 
with respect to prescribed accuracy $\varepsilon$.

Perhaps the most basic operation for an $\mathcal{H}$-tensor $A\in\mathbb{R}^{\mathcal{I}}$ is the evaluation of single
tensor entries for given indices $(i_1,\ldots ,i_d)\in\mathcal{I}$. As we do not store the tensor entry-by-entry,
this is as well an issue. The evaluation has briefly been addressed in Section~\ref{subsection:htuckerformat} and
can also be parallelized level-wise.

\subsection{Evaluation of tensor entries in the Hierarchical Tucker format}\label{subsection:evaluation}
Assume the tensor $A\in\mathbb{R}^{\mathcal{I}}$ to be stored in the Hierarchical Tucker format, based on the dimension
tree $T_{d}$. For a given index $(i_1,\ldots ,i_d)\in\mathcal{I}$ the tensor entry $A_{(i_1,\ldots ,i_d)}$ 
shall be computed. Note that each leaf frame $U_{\{\mu\}}\in\mathbb{R}^{\mathcal{I}_{\mu}\times k_{\{\mu\}}}$ only depends 
on the index $i_{\mu}\in\mathcal{I}_{\mu}$, $\mu = 1,\ldots ,d$. Thus in each leaf frame we just choose the corresponding row
$(U_{\{\mu\}})_{i_{\mu},-}$ and send it to the father node. This can be carried out in parallel for all leaves.

Consider the node $t = \{\mu ,\nu\}$ as father of the leaves $\{\mu\}$ and $\{\nu\}$ in $T_{d}$. 
The compute node belonging to $t$ then receives the rows $(U_{\{\mu\}})_{i_{\mu},-}$ and $(U_{\{\nu\}})_{i_{\nu},-}$ 
from its sons and can thus compute the row $(U_{\{\mu,\nu\}})_{(i_{\mu},i_{\nu}),-}$ of its own frame 
by \eqref{equation:transfer}
and send it to its father
(remember that for an inner node $t$ the frame $U_t$ is not stored explicitly - instead the coefficients $B_{t}$ 
are stored):
\begin{equation*}
(U_t)_{(i_{\mu},i_{\nu}),q} = \sum_{j=1}^{k_{\{\mu\}}}\sum_{\ell = 1}^{k_{\{\nu\}}}
(B_{t})_{q,j,\ell} \cdot (U_{\{\mu\}})_{i_{\mu},j} \cdot  (U_{\{\nu\}})_{i_{\nu},\ell} \quad ,\qquad
q = 1,\ldots ,k_{t}.
\end{equation*}
In general each inner node $t\in T_{d}$ with $\mathrm{sons}(t) = \{s_1,s_2\}$ receives the rows
$(U_{s_1})_{\mathbf{i}_{s_1},-}$ and $(U_{s_2})_{\mathbf{i}_{s_2},-}$ from its sons and then computes the row
$(U_{t})_{\mathbf{i}_{t},-}$ using \eqref{equation:transfer}:
\begin{equation}\label{equation:transfermult}
(U_t)_{\mathbf{i}_t,q} = \sum_{j = 1}^{k_{s_1}}\sum_{\ell = 1}^{k_{s_2}} (B_{t})_{q,j,\ell}\cdot
(U_{s_1})_{\mathbf{i}_{s_1},j}\cdot (U_{s_2})_{\mathbf{i}_{s_2},\ell}\quad ,\qquad q = 1,\ldots ,k_t\; ,
\end{equation}
where $\mathbf{i}_s = (i_{\mu})_{\mu\in s}$ denotes the subindex of $(i_1,\ldots ,i_d)$ corresponding to
the subset $s\subset\{1,\ldots ,d\}$. After computing the row $(U_t)_{\mathbf{i}_t, q}$, it is sent to the father node.

Finally the entry $A_{(i_1,\ldots ,i_d)}$ is obtained on the root node $D = \{1,\ldots ,d\}$ after the rows 
$U_{\mathbf{i}_{t_1},-}$ and $U_{\mathbf{i}_{t_2},-}$ of the root sons $t_1,\,t_2$ have been received:
\begin{equation*}
A_{(i_1,\ldots ,i_d)} = (U_D)_{((i_1,\ldots ,i_d),1)}
= \sum_{j=1}^{k_{t_1}}\sum_{\ell = 1}^{k_{t_2}}(B_D)_{j,\ell} U_{\mathbf{i}_{t_1},j}U_{\mathbf{i}_{t_2},\ell}
\end{equation*}
(remember: the root transfer tensor $B_D$ is only 2-dimensional).

The multiplications in \eqref{equation:transfermult} need $\mathcal{O}(k_t k_{s_1} k_{s_2})$ operations for 
each inner node $t\in T_d$.
Depending on the maximal rank $k:=\max\{k_t\mid t\in T_{d}\}$, we can therefore estimate the complexity
in each node by $\mathcal{O}(k^3)$.

Assuming the index $(i_1,\ldots ,i_d)$ to be known on each leaf node of the tree, 
the process of selecting all the rows $(U_{\{\mu\}})_{i_{\mu},-}$, 
$\mu = 1,\ldots ,d$, in the leaves and sending them to the respective father nodes works independently on each leaf and can
therefore be carried out in parallel on all leaves.

Each inner node $t$ on level $\ell$ of the tree requires only the results of its sons $s_1$ and $s_2$, which are
on level $\ell + 1$. Thus no interaction to any other node on level $\ell$ is required, i.e. all nodes on level $\ell$ work
entirely independent of each other and can act in parallel.

This demonstrates that the evaluation can be parallelized level-wise with respect to 
the tree $T_{d}$ as illustrated in
Fig.~\ref{figure:paralleleval}.

\begin{figure}
\begin{subfigure}[b]{0.48\textwidth}
\begin{tikzpicture}[scale=1.2]
\node (U1)  {\leafframeindexCol{black!15}{blue!50}{2}};
\node (U2)      [right of = U1, xshift = 0.2cm] {\leafframeindexCol{black!15}{red!50}{4}};
\node (U3)      [right of = U2, xshift = 1.0cm] {\leafframeindexCol{black!15}{blue!50}{5}};
\node (U4)      [right of = U3, xshift = 0.2cm] {\leafframeindexCol{black!15}{red!50}{3}};
\node (B12)     [above right of = U1, yshift = 1.0cm] {\transfertensorCol{black!15}} edge (U1) edge (U2);
\node (B34)     [above right of = U3, yshift = 1.0cm] {\transfertensorCol{black!15}} edge (U3) edge (U4);
\node (B1234)   [above right of = B12, xshift = 1.0cm, yshift = 1.0cm] {\roottransfertensorCol{black!15}} edge (B12) edge (B34);
\end{tikzpicture}
\caption{Highest level ($\ell = 2$): In each leaf node $\{\mu\}$, $\mu = 1,2,3 ,4$ choose the row $i_{\mu}$ and send it to the father node.\\ $\,$\\ $\,$}
\label{figure:parallelevala}
\end{subfigure}
\quad
\begin{subfigure}[b]{0.48\textwidth}
\begin{tikzpicture}[scale=1.2]
\node (U1)  {\leafframeCol{black!15}};
\node (U2)      [right of = U1, xshift = 0.2cm] {\leafframeCol{black!15}};
\node (U3)      [right of = U2, xshift = 1.0cm] {\leafframeCol{black!15}};
\node (U4)      [right of = U3, xshift = 0.2cm] {\leafframeCol{black!15}};
\node (B12)     [above right of = U1, yshift = 1.0cm] {\transfertensormultresultCol{black!15}{blue!50}{red!50}{green!50}} edge (U1) edge (U2);
\node (B34)     [above right of = U3, yshift = 1.0cm] {\transfertensormultresultCol{black!15}{blue!50}{red!50}{yellow!50}} edge (U3) edge (U4);
\node (B1234)   [above right of = B12, xshift = 1.0cm, yshift = 1.0cm] {\qquad\qquad\quad\roottransfertensormultCol{black!15}{green!50}{yellow!50}{$A_{(i_1,i_2,i_3,i_4)}$}} edge (B12) edge (B34);
\end{tikzpicture}
\caption{Middle level ($\ell = 1$): Multiply the received rows of the sons along the 1st and the 2nd tensor dimension. Send the result to the father (root level, $\ell = 0$), where the tensor entry $A_{(i_1,i_2,i_3,i_4)}$ is finally computed.}
\label{figure:parallelevalb}
\end{subfigure}
\caption{Parallel evaluation of one tensor entry for a given index $(i_1,i_2,i_3,i_4)\in\mathcal{I}$ ($d=4$) in the
Hierarchical Tucker format.
The operations can be parallelized level-wise (${n_{\mu} = n}$ and ${k_{t} = k}$ for ${\mu=1,2,3,4}$ and ${t\in T_{d}}$).}
\label{figure:paralleleval}
\end{figure}
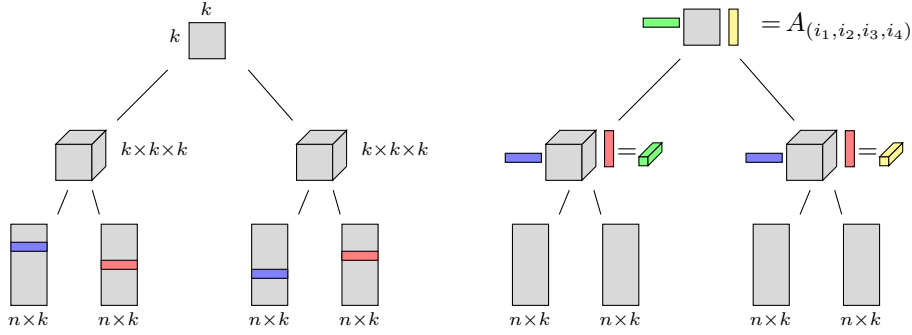

\subsection{The inner product of two $\mathcal{H}$-tensors}\label{subsection:innerproduct} 
The the inner product $\langle A,C\rangle$ of two 
$\mathcal{H}$-tensors 
$A,C\in\mathbb{R}^{\mathcal{I}}$, $\mathcal{I} = \mathcal{I}_1\times\cdots\times\mathcal{I}_d$, 
is defined by \eqref{equation:innerproduct}, i.e. the sizes $n_{\mu} = \#\mathcal{I}_{\mu}$, $\mu = 1,\ldots ,d$, have to
coincide for $A$ and $C$. The tensor ranks, however, can be different for $A$ and $C$, which is why we now write 
$(k_t)_{t\in T_{d}}$ for the Hierarchical ranks of $A$ and $(\bar{k}_{t})_{t\in T_{d}}$ for the
Hierarchical ranks of $C$. Nonetheless we use the same tree $T_{d}$ for $A$ and $C$ and even suppose them to be
distributed on the same compute nodes, i.e. the data of $A$, which belongs to a certain tree node of $T_{d}$ is stored
on the same compute node as the corresponding data of $C$ belonging to the same tree node. The latter assumption can, however,
be left off, which would simply lead to more communication being needed.

As before, we denote the transfer tensors of $A$ by $B_{t}$ for all non-leaf nodes $t\in T_{d}$ and write 
$U_{\{\mu\}}$, $\mu = 1,\ldots ,d$, for the leaf frames, or more general $U_t$ for the frame of node $t\in T_{d}$ 
(these are only stored for the leaves, 
cf. Section~\ref{subsection:htuckerformat}).
For the tensor $C$ we use the notation $F_{t}$ for the transfer
tensors and $V_t$ for the frames.

In order to compute the inner product $\langle A,C\rangle$ in the Hierarchical format we first define
recursively the matrices $\Phi_{t}\in \mathbb{R}^{k_t\times \bar{k}_t}$ for each tree node $t\in T_{d}$, which will coincide with the inner product matrices $(U_{t})^{\top} V_{t}$, $t\in T_{d}$. The inner product will finally be obtained by
$\langle A,C\rangle = (U_D)^{\top}V_D = \Phi_{D}$, where $D = \{1,\ldots ,d\}$ is the root of $T_{d}$.
\begin{definition}\label{definition:phi}
Let $A,C\in\mathbb{R}^{\mathcal{I}}$ be $\mathcal{H}$-tensors, based on the same tree $T_{d}$ with
ranks $(k_t)_{t\in T_{d}}$ and $(\bar{k}_t)_{t\in T_{d}}$, i.e. for $t\in T_{d}$ we have
$U_{t}\in\mathbb{R}^{\mathcal{I}_t\times k_t}$ and $V_t\in\mathbb{R}^{\mathcal{I}_t\times \bar{k}_t}$ (cf. Section~\ref{subsection:htuckerformat}) fulfilling
\begin{equation*}
\range(U_t) = \range(\mathcal{M}_t(A))\qquad\mbox{and}\qquad 
\range(V_t) = \range(\mathcal{M}_t(C)),
\end{equation*}
as well as transfer tensors $B_{t}$ and $F_{t}$ for all non-leaf nodes $t\in T_{d}$ 
fulfilling \eqref{equation:transfer}, which reads
\begin{equation}\label{equation:transfer2}
(V_t)_{-,i} = \sum_{j=1}^{\bar{k}_{s_1}}\sum_{\ell = 1}^{\bar{k}_{s_2}}(F_{t})_{i,j,\ell}\bigl( (V_{s_1})_{-,j}\otimes (V_{s_2})_{-,\ell} \bigr)
\end{equation}
for the tensor $C$, where $\mathrm{sons}(t) = \{s_1,s_2\}$.

The matrices $\Phi_{t}$ are then defined as $\Phi_{\{\mu\}} = (U_{\{\mu\}})^{\top} V_{\{\mu\}}$ for the
leaves, $\mu = 1,\ldots ,d$, and recursively by
\begin{equation}\label{equation:defphi}
(\Phi_{t})_{j_t,\bar{j}_t} = \sum_{j_{s_1}=1}^{k_{s_1}}\sum_{\bar{j}_{s_1}=1}^{\bar{k}_{s_1}}
\sum_{j_{s_2}=1}^{k_{s_2}}\sum_{\bar{j}_{s_2}=1}^{\bar{k}_{s_2}}(B_{t})_{j_t,j_{s_1},j_{s_2}}(\Phi_{s_1})_{j_{s_1},\bar{j}_{s_1}}
(\Phi_{s_2})_{j_{s_2},\bar{j}_{s_2}}(F_{t})_{\bar{j}_t,\bar{j}_{s_1},\bar{j}_{s_2}}.
\end{equation}
for all non-leaf nodes $t\in T_{d}$.
\end{definition}
\begin{lemma}\label{lemma:phi}
For each node $t\in T_{d}$ we have 
\begin{equation}\label{equation:phi}
\Phi_t = (U_t)^{\top}V_t.
\end{equation}
\end{lemma}
\begin{proof}
On the highest level of the tree \eqref{equation:phi} holds by definition:
$\Phi_{\{\mu\}} = (U_{\{\mu\}})^{\top}V_{\{\mu\}}$, $\mu = 1,\ldots ,d$.
Assume \eqref{equation:phi} being true for level $\ell$ of the tree. Then for any node $t\in T_{d}$ of the next
lower level $\ell - 1$ we may use that \eqref{equation:phi} holds true for the sons $s_1$ and $s_2$ of $t$, which are
on level $\ell$. Then \eqref{equation:defphi} yields
\begin{align*}
& (\Phi_{t})_{j_t,\bar{j}_t} \\
& = 
\sum_{\begin{subarray}{l}j_{s_1},\bar{j}_{s_1}\\ j_{s_2},\bar{j}_{s_2}\end{subarray}}
(B_{t})_{j_t,j_{s_1},j_{s_2}}
\biggl( \sum_{\mathbf{i}_{s_1}} (U_{s_1})_{\mathbf{i}_{s_1},j_{s_1}}V_{\mathbf{i}_{s_1},\bar{j}_{s_1}} 
\biggr)
\biggl(
\sum_{\mathbf{i}_{s_2}}(U_{s_2})_{\mathbf{i}_{s_2},j_{s_2}}(V_{s_2})_{\mathbf{i}_{s_2},\bar{j}_{s_2}}
\biggr)
(F_{t})_{\bar{j}_t,\bar{j}_{s_1},\bar{j}_{s_2}} \\
& =
\sum_{\mathbf{i}_{s_1},\mathbf{i}_{s_2}}
\sum_{j_{s_1},j_{s_2}}(B_{t})_{j_t,j_{s_1},j_{s_2}}(U_{s_1})_{\mathbf{i}_{s_1},j_{s_1}}(U_{s_2})_{\mathbf{i}_{s_2},j_{s_2}}
\sum_{\bar{j}_{s_1},\bar{j}_{s_2}} (F_{t})_{\bar{j}_t,\bar{j}_{s_1},\bar{j}_{s_2}} (V_{s_1})_{\mathbf{i}_{s_1},\bar{j}_{s_1}}
(V_{s_2})_{\mathbf{i}_{s_2},\bar{j}_{s_2}} \\
& \underset{\begin{subarray}{l}\eqref{equation:transfer}\\ \eqref{equation:transfer2}\end{subarray}}{=}
\sum_{\mathbf{i}_{s_1},\mathbf{i}_{s_2}} (U_t)_{\mathbf{i}_{t},j_t} (V_t)_{\mathbf{i}_t,\bar{j}_t} 
= \sum_{\mathbf{i}_{t}} (U_t)_{\mathbf{i}_t,j_t}(V_t)_{\mathbf{i}_t,\bar{j}_t} = (U_t)^{\top}V_t,
\end{align*}
where we used the notation $\mathbf{i}_{s} = (i_{\mu})_{\mu\in s}$ for subsets $s\subset \{1,\ldots ,d\}$ again. From this 
\eqref{equation:phi} follows for all levels $\ell$ of $T_{d}$ by induction.
\end{proof}

Using Definition~\ref{definition:phi} and Lemma~\ref{lemma:phi}, we can compute the inner product
$\langle A,C\rangle$ for two $\mathcal{H}$-tensors $A,C\in\mathbb{R}^{\mathcal{I}}$ by recursively computing the matrices
$\Phi_t$ for each node $t\in T_{d}$. The inner product is then obtained on the root node $D = \{1,\ldots ,d\}$ as
$\langle A, C\rangle = \Phi_{D}$. The matrices $\Phi_{t}$, $t\in T_{d}$, are of size $\mathbb{R}^{k_t\times \bar{k}_t}$ and can be 
computed in ${\mathcal{O}(k_{\{\mu\}}\bar{k}_{\{\mu\}} \#\mathcal{I}_{\mu})}$ for each leaf node $\{\mu\}$, 
in ${\mathcal{O}(k_t k_{s_1} k_{s_2}\bar{k}_{s_1} + k_t k_{s_2}\bar{k}_{s_1}\bar{k}_{s_2} + k_{t}\bar{k}_t \bar{k}_{s_1} \bar{k}_{s_2})}$ for each inner node $t$ with
$\mathrm{sons}(t) = \{s_1,s_2\}$ and in 
${\mathcal{O}(k_{t_1}k_{t_2}\bar{k}_{t_1} + k_{t_2}\bar{k}_{t_1}\bar{k}_{t_2} + \bar{k}_{t_1}\bar{k}_{t_2})}$ 
for the root $D = \{1,\ldots ,d\}$, where
$\mathrm{sons}(D) = \{t_1,t_2\}$. With $k :=\max\{k_t, \bar{k}_t\mid t\in T_d\}$ this yields a complexity of
$\mathcal{O}(k^4)$, where the ordering of the multiplications in \eqref{equation:defphi} might
be relevant (in the case of inhomogenous ranks $k_t$, $t\in T_d$).

For each leaf node, the definition $\Phi_{\{\mu\}}=(U_t)^{\top}V_t$, $\mu = 1,\ldots ,d$, requires only data stored 
on the respective leaf $\{\mu\}$; for each non-leaf node $t\in T_{d}$ the definition~\eqref{equation:defphi} 
of $\Phi_t$ requires
data stored on $t$ as well as the matrices $\Phi_{s_1}$, $\Phi_{s_2}$ of the sons.
The computation can therefore be parallelized level-wise, as illustrated in Fig.~\ref{figure:innerproduct}.

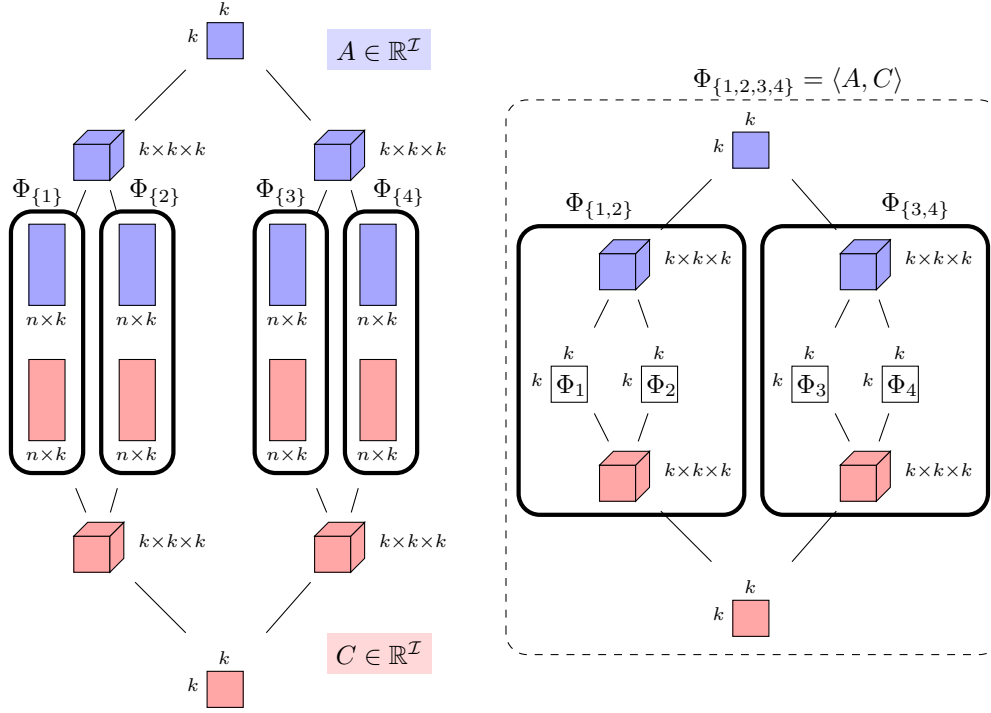
\begin{figure}
\begin{subfigure}[b]{0.48\textwidth}
\begin{tikzpicture}[scale=1.2]
\node (U1)  {\leafframeCol{blue!35}};
\node (U2)      [right of = U1, xshift = 0.2cm] {\leafframeCol{blue!35}};
\node (U3)      [right of = U2, xshift = 1.0cm] {\leafframeCol{blue!35}};
\node (U4)      [right of = U3, xshift = 0.2cm] {\leafframeCol{blue!35}};
\node (B12)     [above right of = U1, yshift = 1.0cm] {\transfertensorCol{blue!35}} edge (U1) edge (U2);
\node (B34)     [above right of = U3, yshift = 1.0cm] {\transfertensorCol{blue!35}} edge (U3) edge (U4);
\node (B1234)   [above right of = B12, xshift = 1.0cm, yshift = 1.0cm] {\roottransfertensorCol{blue!35}} edge (B12) edge (B34);
\node (A)       [right of = B1234, xshift = 1.0cm, yshift = -0.3cm, fill=blue!15] {$A\in\mathbb{R}^{\mathcal{I}}$};
\node (V1)      [below of = U1, yshift = -0.8cm] {\leafframeCol{red!35}};          
\node (V2)      [below of = U2, yshift = -0.8cm] {\leafframeCol{red!35}};          
\node (V3)      [below of = U3, yshift = -0.8cm] {\leafframeCol{red!35}};          
\node (V4)      [below of = U4, yshift = -0.8cm] {\leafframeCol{red!35}};          
\node (F12)     [below right of = V1, yshift = -1.0cm] {\transfertensorCol{red!35}} edge (V1) edge (V2);
\node (F34)     [below right of = V3, yshift = -1.0cm] {\transfertensorCol{red!35}} edge (V3) edge (V4);
\node (F1234)   [below right of = F12, xshift = 1.0cm, yshift = -1.0cm] {\roottransfertensorCol{red!35}} edge (F12) edge (F34);
\node (C)       [right of = F1234, xshift = 1.0cm, yshift = 0.3cm, fill=red!15] {$C\in\mathbb{R}^{\mathcal{I}}$};
\draw[ultra thick, rounded corners = 8pt] (-0.4,-2.1) rectangle (0.4,0.8);
\draw[ultra thick, rounded corners = 8pt] (0.6,-2.1) rectangle (1.4,0.8);
\draw[ultra thick, rounded corners = 8pt] (2.3,-2.1) rectangle (3.1,0.8);
\draw[ultra thick, rounded corners = 8pt] (3.3,-2.1) rectangle (4.1,0.8);
\node (PH1)     at (-0.1,1.01) {$\Phi_{\{1\}}$};
\node (PH2)     at (1.2,1.01) {$\Phi_{\{2\}}$};
\node (PH3)     at (2.6,1.01) {$\Phi_{\{3\}}$};
\node (PH4)     at (3.9,1.01) {$\Phi_{\{4\}}$};
\end{tikzpicture}
\caption{Highest level ($\ell = 2$): On each leaf node $\{\mu\}$, $\mu = 1,2,3,4$, the matrix 
$\Phi_{\{\mu\}} = (U_{\{\mu\}})^{\top}V_{\{\mu\}}$ is computed and sent to the father node(s).\\$\,$}
\label{figure:innerproducta}
\end{subfigure}\quad
\begin{subfigure}[b]{0.48\textwidth}
\begin{tikzpicture}[scale=1.2]
\node (P1)  {\roottransfertensor};
\node (P2)      [right of = P1, xshift = 0.2cm] {\roottransfertensor};
\node (P3)      [right of = P2, xshift = 1.0cm] {\roottransfertensor};
\node (P4)      [right of = P3, xshift = 0.2cm] {\roottransfertensor};
\node (B12)     [above right of = P1, yshift = 0.7cm] {\transfertensorCol{blue!35}} edge (P1) edge (P2);
\node (B34)     [above right of = P3, yshift = 0.7cm] {\transfertensorCol{blue!35}} edge (P3) edge (P4);
\node (B1234)   [above right of = B12, xshift = 1.0cm, yshift = 1.0cm] {\roottransfertensorCol{blue!35}} edge (B12) edge (B34);
\node (F12)     [below right of = P1, yshift = -0.7cm] {\transfertensorCol{red!35}} edge (P1) edge (P2);
\node (F34)     [below right of = P3, yshift = -0.7cm] {\transfertensorCol{red!35}} edge (P3) edge (P4);
\node (F1234)   [below right of = F12, xshift = 1.0cm, yshift = -1.0cm] {\roottransfertensorCol{red!35}} edge (F12) edge (F34);
\node (PH4)     [right of = P4, xshift = -1.02cm, yshift = -0.2cm] {$\Phi_4$};
\node (PH3)     [right of = P3, xshift = -1.02cm, yshift = -0.2cm] {$\Phi_3$};
\node (PH2)    [right of = P2, xshift = -1.02cm, yshift = -0.2cm] {$\Phi_2$};
\node (PH1)    [right of = P1, xshift = -1.02cm, yshift = -0.2cm] {$\Phi_1$};
\draw[ultra thick, rounded corners = 8pt] (-0.6,-1.6) rectangle (1.9,1.6);
\draw[ultra thick, rounded corners = 8pt] (2.1,-1.6) rectangle (4.6,1.6);
\draw[dashed, rounded corners = 8pt] (-0.74,-3.15) rectangle (4.74,3);
\node (PH12)    at (0.3,1.82) {$\Phi_{\{1,2\}}$};
\node (PH34)    at (3.8,1.82) {$\Phi_{\{3,4\}}$};
\node (PH1234)  at (2.5,3.2) {$\Phi_{\{1,2,3,4\}} = \langle A,C\rangle$};
\end{tikzpicture}
\vspace{8mm}
\caption{Middle level ($\ell = 1$): The matrices $\Phi_{\{1,2\}}$ and $\Phi_{\{3,4\}}$ are computed and sent to the root node(s) ($\ell = 0$), 
where the inner product $\langle A, C\rangle = \Phi_{\{1,2,3,4\}}$ is finally computed.}
\label{figure:innerproductb}
\end{subfigure}
\caption{Parallel computation of the inner product $\langle A, C\rangle$ for two $\mathcal{H}$-tensors 
$A,C\in\mathbb{R}^{\mathcal{I}}$ ($d=4$). The operations can be parallelized level-wise.}
\label{figure:innerproduct}
\end{figure}

\subsection{The truncation of $\mathcal{H}$-tensors based on the Hierarchical SVD \cite{Gr10}}\label{subsection:truncation}
The truncation of an $\mathcal{H}$-tensor down to lower
Hierarchical ranks is essential for the addition $A+C$ of two $\mathcal{H}$-tensors
and for the application of some operator $\mathcal{L}$ to $A$,
where the operator
$\mathcal{L}\in\mathbb{R}^{\mathcal{J}\times\mathcal{I}}$ is as well stored in the Hierarchical Tucker format.
Both operations increase the Hierarchical ranks, as it is well known for the addition of 
matrices and the matrix rank from linear algebra. 

In this section we briefly summarize the truncation techniques from \cite{Gr10} to truncate an $\mathcal{H}$-tensor
$A\in\mathbb{R}^{\mathcal{I}}$ down to lower Hierarchical ranks $(k_t)_{t\in T_{d}}$, either for prescribed ranks or
for prescribed accuracy $\varepsilon$. Moreover, we describe how these techniques can be 
parallelized level-wise with respect to
the underlying tree $T_{d}$ for the case of distributed tensors.

In general the truncation of a tensor $A\in\mathbb{R}^{\mathcal{I}}$ (not necessarily an $\mathcal{H}$-tensor)
to prescribed ranks $(k_t)_{t\in T_{d}}$ for
some tensor tree $T_{d}$ can be achieved like this (cf. \cite{Gr10}):
\begin{enumerate}
\item For all leaves $\{\mu\}$, $\mu = 1,\ldots ,d$, compute an SVD $\mathcal{M}_{\{\mu\}}(A) =U\Sigma V^{\top}$ 
\label{enumerate:truncation1}
of the matricization $\mathcal{M}_{\{\mu\}}(A)$ and store 
the left singular vectors corresponding to the $k_{\{\mu\}}$ largest singular values as frame $U_{\{\mu\}}$:
\begin{equation*}
U_{\{\mu\}} = U\mid_{\mathcal{I}_{\mu}\times\{1,\ldots ,k_{\{\mu\}}\}}\qquad \mbox{(singular values in descending order)}.
\end{equation*}
\item For all inner nodes $t\in T_{d}$ with $\mathrm{sons}(t)=\{s_1,s_2\}$ 
compute $U_{t} = U\mid_{\mathcal{I}_t\times\{1,\ldots ,k_t\}}$, 
where $\mathcal{M}_t(A) = U\Sigma V^{\top}$ as in (\ref{enumerate:truncation1}.).
Instead of storing the frame $U_t$ we store the transfer tensor $B_t\in\mathbb{R}^{k_t\times k_{s_1}\times k_{s_2}}$, 
which can be computed by
\begin{equation*}
(B_t)_{i,j,\ell} = \langle (U_t)_{-,i}\; ,\; (U_{s_1})_{-,j}\otimes (U_{s_2})_{-,\ell}\rangle ,
\end{equation*}
where $U_{s_1}$ and $U_{s_2}$ are the respective singular vectors for the sons of $t$.
\item For the root $D = \{1,\ldots ,d\}$ with $\mathrm{sons}(D) = \{t_1,t_2\}$ compute the transfer tensor $B_{D}$ by
\begin{equation*}
(B_D)_{j,\ell} = \langle \mathcal{M}_D(A)\; ,\; (U_{t_1})_{-,j}\otimes (U_{t_2})_{-,\ell}\rangle.
\end{equation*}
\end{enumerate}

In the case of large tensors, which are already given in the Hierarchical Tucker format, 
computing the full matricizations 
$\mathcal{M}_t(A)$ can easliy get impractical. Even the computation of $U_t$ for a non-leaf node $t\in T_{d}$ 
could be no more feasible. 

We make use of results from \cite{Gr10}, which allow for the truncation of an $\mathcal{H}$-tensor $A\in\mathbb{R}^{\mathcal{I}}$
down to smaller Hierarchical ranks, without explicitly computing $\mathcal{M}_t(A)$ or $U_t$ for non-leaf nodes
$t\in T_{d}$. For now we assume the columns of the frames $U_t\in\mathbb{R}^{\mathcal{I}_t\times k_t}$ 
to be orthonormal for all non-root nodes
$t\in T_{d}$ (if this is not the case, the tensor can be orthogonalized, cf. Section~\ref{subsection:orthogonalization}).
For all non-root nodes $t\in T_{d}$ one can find a matrix $V_t\in\mathbb{R}^{\mathcal{I}_{[t]}\times k_t}$, 
$[t] = \{1,\ldots ,d\}\setminus t$, such that
the matricization $\mathcal{M}_t(A)$ can be written as
\begin{equation}\label{equation:matricizationdecomp}
\mathcal{M}_t(A) = U_t(V_t)^{\top},
\end{equation}
since $\range(U_t) = \range(\mathcal{M}_t(A))$. From \cite{Gr10} we know how to compute the matrices
${\widehat{B}_{t} := (V_t)^{\top}V_t\in\mathbb{R}^{k_t\times k_t}}$ for all
$t\in T_{d}\setminus \{D\}$ in $\mathcal{O}(dk^4)$, $k:=\max\{k_t\mid t\in T_d\}$, for
the case that the columns of $U_t$, $t\in T_d\setminus \{D\}$, are orthonormal.
If ${V_t^{\top} = Q_t\Sigma_t W_t^{\top}}$ is a singular value decomposition of $V_t^{\top}$, then
$\widehat{B}_t Q_t = Q_t\Sigma_t^2$, i.e. we get the left singular vectors and singular values of $V_t^{\top}$ as
the eigenvectors and the square roots of the eigenvalues of $\widehat{B}_t$.
Using \eqref{equation:matricizationdecomp} and the fact that the columns of
$U_t$ are orthonormal, we get 
$U_tQ_t\in\mathbb{R}^{\mathcal{I}_t\times k_t}$ as left singular vectors of $\mathcal{M}_t(A)$ with the same 
singular values. We can therefore truncate from rank $k_t$ down to rank $r_t < k_t$
with respect to $t\in T_{d}\setminus\{D\}$ by multiplying $U_t$ only with the first $r_t$ columns of $Q_t$, i.e.
\begin{equation}\label{equation:nodewisetruncation}
U_t \leftarrow U_t\cdot (Q_t\mid_{\{1,\ldots, k_t\}\times \{1,\ldots ,r_t\}}).
\end{equation}
From the following Lemma \ref{lemma:transfertransformation} (cf. \cite{Gr10}) we will see how the transfer tensors have to be transformed in order to achieve
\eqref{equation:nodewisetruncation} for each node $t\in T_{d}\setminus\{D\}$. Since we will use
Lemma \ref{lemma:transfertransformation} again in Section \ref{subsection:orthogonalization}, we
formulate it for arbitrary invertible matrices $Y,Z$:
\begin{lemma}[Transfer tensor transformation]\label{lemma:transfertransformation}
Let $A\in\mathbb{R}^{\mathcal{I}}$ be an $\mathcal{H}$-tensor with underlying tree $T_{d}$ and let
$U_t$ and $B_{t}$ be defined as in Section~\ref{subsection:htuckerformat}, i.e.
\begin{equation*}
(U_t)_{-,i} = \sum_{j=1}^{k_{s_1}}\sum_{\ell = 1}^{k_{s_2}}(B_t)_{i,j,\ell} (U_{s_1})_{-,j}\otimes (U_{s_2})_{-,\ell}
\end{equation*}
for all columns of $U_t$, where $t\in T_{d}$ is a non-leaf node with $\mathrm{sons}(t) = \{s_1,s_2\}$.

Then for matrices $X\in\mathbb{R}^{k_t\times k_t}$, $Y\in\mathbb{R}^{k_{s_1}\times k_{s_1}}$, 
$Z\in\mathbb{R}^{k_{s_2}\times k_{s_2}}$, where $Y,Z$ are invertible, the transformations
$U_t\leftarrow U_t X$, $U_{s_1}\leftarrow U_{s_1} Y$ and $U_{s_2}\leftarrow U_{s_2} Z$ yield the transformed transfer tensor
\begin{equation*}
B_{t,\mathrm{new}} = (X^{\top},Y^{-1},Z^{-1})\circ B_t,\quad i.e.\quad (B_{t,\mathrm{new}})_{i,j,k}
= \sum_{\ell,m,n} X_{\ell,i}(Y^{-1})_{j,m}(Z^{-1})_{k,n}(B^{t})_{\ell, m, n}\, ,
\end{equation*}
such that
\begin{equation*}
(U_t X)_{-,i} = \sum_{j=1}^{k_{s_1}}\sum_{\ell = 1}^{k_{s_2}} (B_{t,\mathrm{new}})_{i,j,\ell}(U_{s_1}Y)_{-,j}\otimes
(U_{s_2}Z)_{-,\ell}.
\end{equation*}
\end{lemma}

Let us first suppose, we transform $U_t\leftarrow U_tQ_t$ in each non-root node $t\in T_{d}$, where 
$Q_t$ is the full orthogonal matrix of eigenvectors for $t\in T_{d}$ from above, i.e. 
$Q_t^{-1} = Q_t^{\top}$.
Then the transfer tensors would transform like
\begin{equation}\label{equation:transfertransformation}
B_{t,\mathrm{new}} = (Q_t^{\top},Q_{s_1}^{\top},Q_{s_2}^{\top})\circ B_t,
\end{equation}
according to Lemma~\ref{lemma:transfertransformation}, where $\mathrm{sons}(t) = \{s_1, s_2\}$. Truncating the tensor down to
ranks $(r_t)_{t\in T_{d}}$ means
leaving only the first $r_t$ columns of $Q_t$ in each non-root node $t\in T_{d}$ 
and removing the remaining ones (assuming the eigenvalues/singular values to be in descending order).
This can directly be carried out in \eqref{equation:transfertransformation} by first truncating the matrices
$Q_t$, $Q_{s_1}$, $Q_{s_2}$ and multiplying afterwards:
\begin{align}
& B_{t,\mathrm{new}} =\notag \\
& \bigl((Q_t\mid_{\{1,\ldots ,k_t\}\times \{1,\ldots ,r_t\}})^{\top},
(Q_{s_1}\mid_{\{1,\ldots ,k_{s_1}\}\times \{1,\ldots ,r_{s_1}\}})^{\top},
(Q_{s_2}\mid_{\{1,\ldots ,k_{s_2}\}\times \{1,\ldots ,r_{s_2}\}})^{\top}\bigr)\circ B_t.
\label{equation:transfertransformationtrunc}
\end{align}

For the two sons $t_1$ and $t_2$ of the root $D = \{1,\ldots ,d\}$, the matrices 
$\widehat{B}_t = (V_t)^{\top}V_t\in\mathbb{R}^{k_t\times k_t}$ are just
\begin{equation}\label{equation:acctransferroot}
\widehat{B}_{t_1} = B_{D}(B_{D})^{\top}\qquad\mbox{and}\qquad
\widehat{B}_{t_2} = (B_{D})^{\top}B_{D}.
\end{equation}
For the sons $s_1$ and $s_2$ of a non-root node $t$, the computation of $\widehat{B}_{s_1}$ and
$\widehat{B}_{s_2}$ involves the matrix $\widehat{B}_t$ as well as the transfer tensor
$B_t$ (cf. \cite{Gr10}):
\begin{align}\label{equation:acctransferSon1}
(\widehat{B}_{s_1})_{p,q} &= \sum_{i_2=1}^{k_t}\sum_{\ell =1}^{k_{s_2}}\sum_{i_1=1}^{k_t}
(B_t)_{i_1,p,\ell}(\widehat{B}_t)_{i_1,i_2}(B_t)_{i_2,q,\ell} \quad \mbox{for}\quad
p,q = 1,\ldots ,k_{s_1}\quad\mbox{and} \\
\label{equation:acctransferSon2}
(\widehat{B}_{s_2})_{p,q} &= \sum_{i_2=1}^{k_t}\sum_{j=1}^{k_{s_1}}\sum_{i_1=1}^{k_t}
(B_t)_{i_1,j,p}(\widehat{B}_t)_{i_1,i_2}(B_t)_{i_2,j,q}\quad\mbox{for}\quad
p,q = 1,\ldots ,k_{s_2}\; .
\end{align}
Notice that the equations from \eqref{equation:acctransferroot} can be written as 
\eqref{equation:acctransferSon1} and \eqref{equation:acctransferSon2} if we set
$\widehat{B}_D = 1\in\mathbb{R}^{1\times 1}$.

The computation
of the matrices $\widehat{B}_{t}$ can thus be parallelized level-wise with respect to the tree 
$T_{d}$, where the computations start
at the root node, which computes both matrices from \eqref{equation:acctransferroot} and sends them to its sons $t_1$ and $t_2$.
Subsequently any non-root node $t\in T_{d}$ receives the matrix $\widehat{B}_{t}$ from its father and,
if $t$ is not a leaf, computes the matrices $\widehat{B}_{s_1}$ and $\widehat{B}_{s_2}$ for its sons $s_1$ and $s_2$ according to \eqref{equation:acctransferSon1} and \eqref{equation:acctransferSon2}.

Once the matrices $\widehat{B}_{t}$ are computed for all non-root nodes $t\in T_{d}$, the truncation can be 
started on the highest level of $T_{d}$. Again, all nodes of one level work entirely independent of each other:
\begin{enumerate}
\item For all leaves $\{\mu\}$, $\mu = 1,\ldots ,d$: Compute the eigenvectors $Q_{\{\mu\}}$ of $\widehat{B}_{\{\mu\}}$
together with the corresponding eigenvalues. Overwrite the frames $U_{\{\mu\}}$ with the truncated singular vectors:
$U_{\{\mu\}}\leftarrow U_{\{\mu\}}\bigl(Q_{\{\mu\}}\mid _{\{1,\ldots ,k_{\{\mu\}}\}\times\{1,\ldots ,r_{\{\mu\}}\}}\bigr)$.
Send the truncated eigenvector matrix $Q_{\{\mu\}}\mid _{\{1,\ldots ,k_{\{\mu\}}\}\times\{1,\ldots ,r_{\{\mu\}}\}}$ to the father.
\item For all inner nodes $t\in T_{d}$ with $\mathrm{sons}(t) = \{s_1,s_2\}$ receive the truncated matrices
$Q_{s_1}\mid_{\{1,\ldots ,k_{s_1}\}\times \{1,\ldots ,r_{s_1}\}}$ and 
$Q_{s_2}\mid_{\{1,\ldots ,k_{s_2}\}\times \{1,\ldots ,r_{s_2}\}}$
from the sons $s_1$ and $s_2$, compute $Q_t$, truncate it to
$Q_t\mid_{\{1,\ldots ,k_t\}\times \{1,\ldots ,r_{t}\}}$, send the truncated matrix to the father, 
and update the transfer tensor $B_{t}$ according to 
\eqref{equation:transfertransformationtrunc}.
\item For the root $D = \{1,\ldots ,d\}$: Receive the truncated matrices
$Q_{t_1}\mid_{\{1,\ldots ,k_{t_1}\}\times \{1,\ldots ,r_{t_1}\}}$ and 
$Q_{t_2}\mid_{\{1,\ldots ,k_{t_2}\}\times \{1,\ldots ,r_{t_2}\}}$
from the sons $t_1$, $t_2$ of $D$ and transform $B_D$ like
\begin{equation*}
B_D\leftarrow Q_{t_1}^{\top}B_D Q_{t_2}.
\end{equation*}
\end{enumerate}
This shows that the truncation of $\mathcal{H}$-tensors down to smaller 
ranks can be parallelized level-wise. 

Note that we assumed the columns of $U_t$ to be orthonormal for all non-root nodes $t\in T_{d}$
in order to get the matrices $\widehat{B}_t$ for each non-root node $t$ by 
\eqref{equation:acctransferroot}, \eqref{equation:acctransferSon1} and 
\eqref{equation:acctransferSon2} (cf. \cite{Gr10}).
Furthermore we can only take $U_t Q_t$ as left singular vectors of $\mathcal{M}_t(A)$ if the columns of $U_t$ are orthonormal. If this is not fulfilled, the tensor must be 
orthogonalized first, cf. Section~\ref{subsection:orthogonalization}.

Further notice that after a truncation the columns of $U_t$, $t\in T_d\setminus \{D\}$, 
will probably no longer be
orthonormal (one can easily construct counterexamples where $B_{t,\mathrm{new}}$ from
\eqref{equation:transfertransformationtrunc} does not fulfill the respective condition, i.e. the
columns of $\mathcal{M}_{2,3}(B_{t,\mathrm{new}})$ are not orthonormal). 
This, however, does not affect the truncation algorithm which first computes 
the matrices $\widehat{B}_t$ and the corresponding eigenvectors $Q_t$ 
for all non-root nodes $t\in T_d$
and afterwards starts truncating.
Imagine the truncation to be processed from the root to the leaves,
then each non-root node $t\in T_d$ is first truncated, before its orthogonality gets
destroyed by the transformations induced by its sons. 
Since the order of the multiplications in
\eqref{equation:transfertransformationtrunc} does not play a role, we may also start the truncations
 at the leaf nodes.

\subsection{The orthogonalization of $\mathcal{H}$-tensors}\label{subsection:orthogonalization}
For the truncation of an $\mathcal{H}$-tensor $A\in\mathbb{R}^{\mathcal{I}}$ down to smaller Hierarchical 
ranks, 
as described in Section~\ref{subsection:truncation}, the frames 
$U_t\in\mathbb{R}^{\mathcal{I}_t\times k_t}$ are required to have orthonormal columns for all non-root nodes 
$t$ of the underlying dimension tree $T_{d}$. An $\mathcal{H}$-tensor $A\in\mathbb{R}^{\mathcal{I}}$ fulfilling this property, 
is called \textit{orthogonal}:
\begin{definition}[Orthogonal frames, orthogonal $\mathcal{H}$-tensor]
Let $A\in\mathbb{R}^{\mathcal{I}}$ be an $\mathcal{H}$-tensor with underlying dimension tree $T_{d}$.

A frame $U_t$ is called orthogonal, if its columns are orthonormal.

The $\mathcal{H}$-tensor $A$ is called orthogonal,
if each non-root frame $U_t$, $t\in T_{d}\setminus\{\{1,\ldots ,d\}\}$ is orthogonal.
\end{definition}

For the leaves $\{\mu\}$, $\mu = 1,\ldots ,d$, the frames $U_{\{\mu\}}$ are stored explicitly, i.e. we can directly compute a
reduced QR-decomposition $U_{\{\mu\}} = Q_{\{\mu\}}R_{\{\mu\}}$, keep the orthogonal factor $Q_{\{\mu\}}$ and adjust the transfer
tensor of the father according to Lemma~\ref{lemma:transfertransformation}. For $t = \{\mu ,\nu\}$ we do
\begin{equation*}
U_{\{\mu\}}\leftarrow Q_{\{\mu\}} \;\bigl(=U_{\{\mu\}}R_{\{\mu\}}^{-1}\bigr)\qquad\mbox{and}\qquad
U_{\{\nu\}}\leftarrow Q_{\{\nu\}} \;\bigl(=U_{\{\nu\}}R_{\{\nu\}}^{-1}\bigr)
\end{equation*}
and update
\begin{equation*}
B_{t}\leftarrow (I,R_{\{\mu\}},R_{\{\nu\}})\circ B_t,
\end{equation*}
where $I\in\mathbb{R}^{k_t\times k_t}$ is the unit matrix.

Now consider an inner node $t\in T_{d}$, where the frames $U_{s_1}$ and $U_{s_2}$ of the 
sons $s_1$, $s_2$ have already been orthogonalized. Then $U_t$ being orthogonal is equivalent to 
\begin{equation*}
\sum_{j = 1}^{k_{s_1}}\sum_{\ell = 1}^{k_{s_2}}(B_{t})_{i,j,\ell}(B_t)_{i^{\prime},j,\ell} = \delta_{i,i^{\prime}}
\qquad\mbox{with the Kronecker delta}\;\delta_{i,i^{\prime}},
\end{equation*}
i.e. the orthogonality of the matrix $\mathcal{M}_{\{2,3\}}(B_{t})$ (cf. Definition~\ref{definition:matricization}).
For the inner nodes $t\in T_{d}$ we can thus compute a reduced QR-decomposition 
of the matrix $\mathcal{M}_{\{2,3\}}(B_t)$, i.e.
\begin{equation}\label{equation:qrtransfer}
(B_{t})_{i,j,\ell} = \sum_{q = 1}^{k} (Q_t)_{(j,\ell),q}(R_t)_{q,i}\; ,\qquad\mbox{where}\quad k = 
\min\{k_t\, ,\,k_{s_1}\cdot k_{s_2}\},
\end{equation}
and then update the transfer tensor $B_{f}$ of the father according to Lemma~\ref{lemma:transfertransformation}.

Altogether we have the following algorithm for the orthogonalization of an $\mathcal{H}$-tensor $A\in\mathbb{R}^{\mathcal{I}}$:
\begin{enumerate}
\item For all leaves $\{\mu\}$, $\mu = 1,\ldots ,d$: Compute a reduced QR-decomposition $U_{\{\mu\}} = Q_{\{\mu\}}R_{\{\mu\}}$
and overwrite $U_{\{\mu\}}\leftarrow Q_{\{\mu\}}$. Send the right factor $R_{\{\mu\}}$ to the father.
\item For all inner nodes $t\in T_{d}$ with $\mathrm{sons}(t) = \{s_1,s_2\}$ receive the factors
$R_{s_1}$ and $R_{s_2}$ from the sons and update the transfer tensor: $B_{t}\leftarrow (I,R_{s_1},R_{s_2})\circ B_{t}$.
Then compute a reduced QR-decomposition of $\mathcal{M}_{\{2,3\}}(B_{t})$ like \eqref{equation:qrtransfer} and store
$(B_t)_{q,j,\ell}\leftarrow (Q_t)_{(j,\ell),q}$, with the indices named as in \eqref{equation:qrtransfer}.
The factor $R_t$ is sent to the father.
\item For the root $D = \{1,\ldots ,d\}$: Receive the factors $R_{t_1}$, $R_{t_2}$ of the sons $t_1$, $t_2$ and update the
transfer tensor $B_{D} \leftarrow R_{t_1}B_{D}R_{t_2}^{\top}$.
\end{enumerate}
During the algorithm each non-leaf node $t\in T_{d}$ requires the right factors $R_{s_1}$ and $R_{s_2}$ of its
sons, but works independently of all other nodes on its level. Therefore the orthogonalization of an $\mathcal{H}$-tensor 
can be parallelized level-wise with respect to the underlying tree $T_{d}$.

Depending on the maximal rank $k:=\max\{k_t\mid t\in T_d\}$, the complexity on each node can be estimated by
$\mathcal{O}(k^4)$, since the complexity of the QR decomposition of $\mathcal{M}_{\{2,3\}}(B_t)$ can be
estimated by $\mathcal{O}(k_t^2 k_{s_1} k_{s_2})$ and the complexity of the multiplication
$B_t\leftarrow (I,R_{s_1},R_{s_2})\circ B_t$ can be estimated by 
$\mathcal{O}(k_t k_{s_1}^2 k_{s_2}  + k_t k_{s_1}k_{s_2}^2 )$.
However, in our runtime tests the complexity seems to behave more like $\mathcal{O}(k^3)$, which might be due
to optimizations in the LAPACK routines which we are using.

\subsection{The addition of two $\mathcal{H}$-tensors}\label{subsection:addition}
For the sum $A+C$ of two $\mathcal{H}$-tensors $A,C\in\mathbb{R}^{\mathcal{I}}$, based on the same dimension tree 
$T_{d}$,
\begin{equation*}
(A+C)_{(i_1,\ldots ,i_d)} = A_{(i_1,\ldots ,i_d)} + C_{(i_1,\ldots ,i_d)}
\qquad\mbox{for}\quad (i_1,\ldots ,i_d)\in\mathcal{I}\;,
\end{equation*}
we assume $A$ and $C$ to be distributed among the same compute nodes, i.e. the data of $A$, which belongs to a certain node
of the underlying tree $T_{d}$, is stored on the same compute node as the corresponding data of $C$.
Due to this assumption we do not need any communication for the assembling of the tensor $A+C$ in the Hierarchical format.
Otherwise the respective nodes would have to communicate. The assembling of $A+C$ by itself (without subsequent truncation) 
does not even require any 
arithmetic operations, as the following Lemma~\ref{lemma:sum} illustrates, a proof of which can
e.g. be found in \cite{Ha12}.
\begin{lemma}[The sum of two $\mathcal{H}$-tensors]\label{lemma:sum}
Let $A\in\mathbb{R}^{\mathcal{I}}$ and $C\in\mathbb{R}^{\mathcal{I}}$ be two $\mathcal{H}$-tensors, based on the same dimension
tree $T_{d}$. For $A$ we use the previous notation $B_{t}$ for the transfer tensors and
$U_t$ for the frames. The respective objects of $C$ are denoted by $F_{t}$ and $V_t$. 
Furthermore we write $(k_t)_{t\in T_{d}}$ for the Hierarchical ranks of $A$ and $(\bar{k}_t)_{t\in T_{d}}$
for those of $C$. For a non-leaf node $t$ we write $\mathrm{sons}(t) = \{s_1,s_2\}$.

Then the sum $A+C$ can again be represented in the Hierarchical Tucker format with transfer tensors 
$H_{t}\in\mathbb{R}^{(k_{t} + \bar{k}_t)\times (k_{s_1} + \bar{k}_{s_1})\times (k_{s_2} + \bar{k}_{s_2})}$ 
for all inner nodes $t\in T_{d}$,
$H_{D}\in\mathbb{R}^{(k_{t_1}+\bar{k}_{t_1})\times (k_{t_2}+\bar{k}_{t_2})}$ for the root $D = \{1,\ldots ,d\}$, where $\mathrm{sons}(D) = \{t_1,t_2\}$,
and leaf frames
$W_{\{\mu\}}\in\mathbb{R}^{\mathcal{I}_{\{\mu\}}\times (k_{\{\mu\}} + \bar{k}_{\{\mu\}})}$, $\mu = 1,\ldots ,d $, 
defined as follows:
\begin{align} \label{equation:sumleafframes}
(W_{\{\mu\}})_{-,j} & =
\left\{
\begin{array}{ccl}
(U_{\{\mu\}})_{-,j} & , & 1\leq j\leq k_{\{\mu\}} \\
(V_{\{\mu\}})_{-,(j-k_{\{\mu\}})} & , & k_{\{\mu\}} + 1\leq j\leq k_{\{\mu\}} + \bar{k}_{\{\mu\}}
\end{array}
\right.
\mbox{for}\; \mu = 1,\ldots ,d, 
\\ \notag
(H_t)_{i,j,\ell} & =
\left\{
\begin{array}{ccl}
(B_t)_{i,j,\ell} & , & 
\left\{
\begin{array}{l}
1\leq i\leq k_t\; , \\
1\leq j\leq k_{s_1}\; , \\
1\leq \ell\leq k_{s_2} 
\end{array}
\right. \\
(F_t)_{i-k_t, j-k_{s_1}, \ell - k_{s_2}} & , & 
\left\{
\begin{array}{l}
k_t +1\leq i\leq k_t + \bar{k}_t\; ,\\
k_{s_1}+1\leq j\leq k_{s_1} + \bar{k}_{s_1}\; , \\
k_{s_2}+1 \leq \ell\leq k_{s_2} + \bar{k}_{s_2}
\end{array}
\right. \\
0 & , & \mbox{otherwise}
\end{array}
\right. t\in T_{d}\;\mbox{inner node,} \\ \notag
(H_D)_{j,\ell} &= 
\left\{
\begin{array}{ccl}
(B_D)_{j,\ell} & , & 1\leq j\leq k_{t_1}\; \mbox{and}\;\; 1\leq \ell\leq k_{t_2} \\
(F_D)_{j-k_{t_1},\ell -k_{t_2}} & , & k_{t_1}+1\leq j\leq k_{t_1}+\bar{k}_{t_1} \; \mbox{and}\;\;
  k_{t_2} + 1\leq \ell\leq k_{t_2}+\bar{k}_{t_2}\, \\
  0 & ,  & \mbox{otherwise}.
\end{array}
\right.
\end{align}
\end{lemma}

According to Lemma~\ref{lemma:sum} the sum $A+C$ is built by just putting both $\mathcal{H}$-tensors $A$ and $C$ together
(cf. Fig.~\ref{figure:sum}). This process can be run in parallel on each node of the tree and does not need any communication,
under the above assumption.
\begin{figure}
\begin{tikzpicture}[scale=1.2]
\node (U1)  {\leafframeCol{blue!35}};
\node (U2)      [right of = U1, xshift = 0.2cm] {\leafframeCol{blue!35}};
\node (U3)      [right of = U2, xshift = 1.0cm] {\leafframeCol{blue!35}};
\node (U4)      [right of = U3, xshift = 0.2cm] {\leafframeCol{blue!35}};
\node (B12)     [above right of = U1, yshift = 1.0cm] {\transfertensorCol{blue!35}} edge (U1) edge (U2);
\node (B34)     [above right of = U3, yshift = 1.0cm] {\transfertensorCol{blue!35}} edge (U3) edge (U4);
\node (B1234)   [above right of = B12, xshift = 1.0cm, yshift = 1.0cm] {\roottransfertensorCol{blue!35}} edge (B12) edge (B34);
\node (A)       [right of = B1234, xshift = 1.0cm, yshift = -0.3cm, fill=blue!15] {$A\in\mathbb{R}^{\mathcal{I}}$};
\node (V1)      [right of = U4, xshift = 1.5cm] {\leafframeCol{red!35}};
\node (V2)      [right of = V1, xshift = 0.2cm] {\leafframeCol{red!35}};
\node (V3)      [right of = V2, xshift = 1.0cm] {\leafframeCol{red!35}};
\node (V4)      [right of = V3, xshift = 0.2cm] {\leafframeCol{red!35}};
\node (F12)     [above right of = V1, yshift = 1.0cm] {\transfertensorCol{red!35}} edge (V1) edge (V2);
\node (F34)     [above right of = V3, yshift = 1.0cm] {\transfertensorCol{red!35}} edge (V3) edge (V4);
\node (F1234)   [above right of = F12, xshift = 1.0cm, yshift = 1.0cm] {\roottransfertensorCol{red!35}} edge (F12) edge (F34);
\node (C)       [right of = F1234, xshift = 1.0cm, yshift = -0.3cm, fill=red!15] {$C\in\mathbb{R}^{\mathcal{I}}$};
\node (plus)    [right of = B34, xshift = 1.3cm] {\Huge $+$};
\node (W1)      [below right of = U4, xshift = -2.3cm, yshift = -5.0cm] {\leafframeAdd{blue!35}{red!35}};
\node (W2)      [right of = W1, xshift = 0.5cm] {\leafframeAdd{blue!35}{red!35}};
\node (W3)      [right of = W2, xshift = 1.3cm] {\leafframeAdd{blue!35}{red!35}};
\node (W4)      [right of = W3, xshift = 0.5cm] {\leafframeAdd{blue!35}{red!35}};
\node (H12)     [above right of = W1, yshift = 1.5cm] {\transfertensorAdd{blue!35}{red!35}} edge (W1) edge (W2);
\node (H34)     [above right of = W3, yshift = 1.5cm] {\transfertensorAdd{blue!35}{red!35}} edge (W3) edge (W4);
\node (H1234)   [above right of = H12, xshift = 1.25cm, yshift = 1.4cm] {\roottransfertensorAdd{blue!35}{red!35}} edge (H12) edge (H34);
\node (equals)    [left of = H12, xshift = -1.5cm] {\Huge $=$};
\end{tikzpicture}
\caption{The sum of two $\mathcal{H}$-tensors $A,C\in\mathbb{R}^{\mathcal{I}}$ is assembled by merging $A$ and $C$ together,
 cf. Lemma~\ref{lemma:sum}. The white spaces in the transfer tensors represent zero blocks. 
 The sum can be assembled in parallel on each node of the tree.}
\label{figure:sum}
\end{figure}
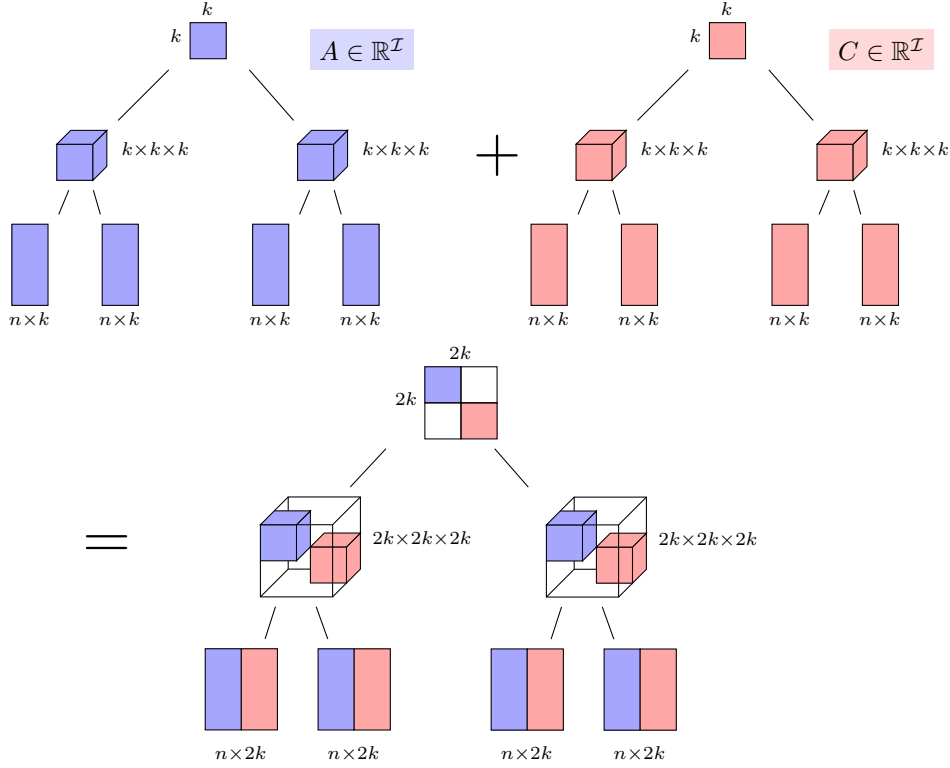
In the end, since we do not have to compute anything for the addition, we can of course let the tensors $A$ and $C$ be
stored, as they are, and just modify the evaluation routine to evaluate $A+C$ for a given index $(i_1,\ldots ,i_d)$ 
instead of $A$ or $C$.

However, if a large number of additions is needed (eg. for some iterative method), we would like to truncate the sum back
to lower ranks (cf. Section~\ref{subsection:truncation}), to save storage. This can be done either 
with respect to prescribed
ranks $(k_t)_{t\in T_{d}}$ or for prescribed accuracy $\varepsilon$ of the truncated tensor.

\subsection{The application of an operator to an $\mathcal{H}$-tensor}\label{subsection:operatorApp}
Let $\mathcal{I} = \mathcal{I}_1\times\cdots\times\mathcal{I}_d$ and $\mathcal{J} = \mathcal{J}_1\times\cdots\times\mathcal{J}_d$
be two product index sets of dimension $d\in\mathbb{N}$ and consider a tensor $A\in\mathbb{R}^{\mathcal{I}}$ as well as a matrix
$\mathcal{L}\in\mathbb{R}^{\mathcal{J}\times\mathcal{I}}$. The matrix $\mathcal{L}$ can as well be seen as a tensor of dimension
$d$ by ordering the dimensions like
\begin{equation}\label{equation:operatorastensor}
\mathcal{L}\in\mathbb{R}^{(\mathcal{J}_1\times\mathcal{I}_1)\times\cdots \times(\mathcal{J}_d\times\mathcal{I}_d)}\; ,\quad\mbox{i.e.}\quad
\mathcal{L}_{(j_1,i_1),\ldots ,(j_d,i_d)}\quad\mbox{instead of}\quad \mathcal{L}_{(j_1,\ldots ,j_d),(i_1,\ldots ,i_d)}\; ,
\end{equation}
and regarding each of the indices $(j_{\mu},i_{\mu})$, $\mu = 1,\ldots ,d$, as one index.
We assume both tensors $\mathcal{L}$ and $A$ to be represented in the Hierarchical Tucker format and call them
$\mathcal{H}$\textit{-operator} and $\mathcal{H}$-tensor.
Then the tensor $C:=\mathcal{L}A\in\mathbb{R}^{\mathcal{J}}$, i.e.
\begin{equation}\label{equation:operator}
C_{(j_1,\ldots ,j_d)} = \sum_{(i_1,\ldots ,i_d)}\mathcal{L}_{(j_1,i_1),\ldots ,(j_d,i_d)}A_{(i_1,\ldots ,i_d)}\; ,
\end{equation}
can be computed directly in the Hierarchical Tucker format. 
The following Lemma~\ref{lemma:operator} describes the representation of the $\mathcal{H}$-tensor $C:=\mathcal{L}A$. A proof can be found in \cite{Ha12}.
\begin{lemma}[The application of an $\mathcal{H}$-operator to an $\mathcal{H}$-tensor]\label{lemma:operator}
Let $\mathcal{I}=\mathcal{I}_1\times\cdots\times\mathcal{I}_d$ and $\mathcal{J}=\mathcal{J}_1\times\cdots\times\mathcal{J}_d$ 
be two product index sets of dimension $d$ and let $A\in\mathbb{R}^{\mathcal{I}}$ be an $\mathcal{H}$-tensor and 
$\mathcal{L}\in\mathbb{R}^{\mathcal{J}\times \mathcal{I}}$ an $\mathcal{H}$-operator. Regarding $\mathcal{L}$ as well as
tensor of dimension $d$, cf. \eqref{equation:operatorastensor}, we suppose $A$ and $\mathcal{L}$ to be both represented in the
Hierarchical Tucker format with the same underlying tree $T_d$. 
For $A$ we use the previous notation $B_t$ for the transfer tensors and $U_t$ for the frames. The respective objects of
$\mathcal{L}$ are denoted by $H_t$ and $W_t$. Furthermore we write $(k_t)_{t\in T_{d}}$ for the Hierarchical ranks of
of $A$ and $(\bar{k}_t)_{t\in T_d}$ for those of $\mathcal{L}$. For a non-leaf node $t$ we 
write $\mathrm{sons}(t) = \{s_1,s_2\}$.

Then the tensor $C:=\mathcal{L}A\in\mathbb{R}^{\mathcal{J}}$ can again be represented in the Hierarchical Tucker format
with transfer tensors $F_t\in\mathbb{R}^{k_t\bar{k}_t\times k_{s_1}\bar{k}_{s_1}\times k_{s_2}\bar{k}_{s_2}}$ for all
non-leaf nodes $t\in T_{d}$ ($k_D = \bar{k}_D = 1$ for the root $D$), and leaf frames 
$V_{\{\mu\}}\in\mathbb{R}^{\mathcal{J}_{\mu}\times k_{\{\mu\}}\bar{k}_{\{\mu\}}}$, $\mu = 1,\ldots ,d$, defined as follows:
\begin{align}\notag
F_{t} & = H_{t}\otimes B_{t} \; ,\quad\mbox{i.e.}\quad (F_t)_{(p,i),(q,j),(r,\ell)} = (H_t)_{p,q,r}(B_t)_{i,j,\ell} \\ \notag
 & \quad (i = 1,\ldots ,k_t,\; j = 1,\ldots ,k_{s_1},\; \ell = 1,\ldots ,k_{s_2},\; \\ \notag
 &\;\quad p = 1,\ldots ,\bar{k}_{t},\; q = 1,\ldots ,\bar{k}_{s_1},\; r = 1,\ldots ,\bar{k}_{s_2}) \qquad\mbox{and}\\ 
 \label{equation:operatorleafframes}
(V_{\{\mu\}})_{j_{\mu},(r ,\ell)} & = \sum_{i_{\mu}\in\mathcal{I}_{\mu}}(W_{\{\mu\}})_{(j_{\mu},i_{\mu}),r}\cdot (U_{\{\mu\}})_{i_{\mu},\ell} \\ \notag
 &\quad (\ell = 1,\ldots ,k_{\{\mu\}} ,\; r = 1,\ldots ,\bar{k}_{\{\mu\}}).
\end{align}
\end{lemma}

From Lemma~\ref{lemma:operator} it follows that the product $\mathcal{L}A$ of an $\mathcal{H}$-operator $\mathcal{L}$ and an
$\mathcal{H}$-tensor $A$ can be computed in parallel on each node $t\in T_d$ without any communication needed, assuming
all data of $A$ and $\mathcal{L}$ belonging to one tree node $t\in T_d$ to be stored on the same compute node.

Since the Hierarchical ranks of $C:=\mathcal{L}A$ are the products $k_t\bar{k}_t$ of the respective ranks $k_t$ and 
$\bar{k}_t$ of $A$ and $\mathcal{L}$, one would most likely want to truncate the tensor $C$ down to given ranks 
$(r_t)_{t\in T_d}$ or 
with respect to some prescribed accuracy $\varepsilon$.

Computing the product $F_t = H_t\otimes B_t$ (cf. Lemma~\ref{lemma:operator}) needs
$\mathcal{O}(k_t \bar{k}_t k_{s_1} \bar{k}_{s_1} k_{s_2}\bar{k}_{s_2})$ operations, i.e. the complexity 
can be estimated by $\mathcal{O}(k^6)$, $k:=\max\{k_t, \bar{k}_t\mid t\in T_d\}$, if the ranks of the 
$\mathcal{H}$-tensor $A$ and the $\mathcal{H}$-operator $\mathcal{L}$ are comparable in size.

\section{Runtime tests}\label{section:runtime}
In this section we present some runtime tests. On the one hand we analyze the parallel runtime of our algorithms 
for varying tensor dimension $d$.
Here we leave the size $n$ of each tensor dimension as well as the rank $k$ 
($k_t = k$ for all $t\in T_d\setminus\{1,\ldots ,d\}$) unchanged. In this case we expect
the parallel runtime to grow like $\mathcal{O}(\log(d))$ for the evaluation of tensor entries, the inner product of two 
$\mathcal{H}$-tensors, the orthogonalization and the truncation down to lower rank of an $\mathcal{H}$-tensor.

For the application of an 
$\mathcal{H}$-operator to an $\mathcal{H}$-tensor we would even expet the runtime to be independent of $d$, since the
algorithms can run in parallel on each node (cf. Section~\ref{section:arithmetic}).

Since the addition of two $\mathcal{H}$-tensors (without subsequent truncation) does not involve any floating point operation
at all (cf. Section~\ref{subsection:addition}), we do not cover it here. The addition of two $\mathcal{H}$-tensors together
with the truncation behaves like the truncation itself.

Furthermore we analyze the dependence of the parallel runtime on the tensor rank $k$. For that we leave $d$ and $n$ unchanged
while we vary the rank $k$. 
Here we expect an $\mathcal{O}(k^4)$ dependence for the inner product (cf. Section~\ref{subsection:innerproduct}), for the truncation of an orthogonalized $\mathcal{H}$-tensor (cf. Section~\ref{subsection:truncation}),
as well as for the orthogonalization of an $\mathcal{H}$-tensor 
(cf. Section~\ref{subsection:orthogonalization}).
For the evaluation of tensor entries we expect an $\mathcal{O}(k^3)$ complexity (cf. Section~\ref{subsection:evaluation}).
For the application of an $\mathcal{H}$-operator we expect $\mathcal{O}(k^6)$,
where also the ranks of the $\mathcal{H}$-operator are chosen to equal $k$ on each node except for the root
node (cf. Sections~\ref{subsection:operatorApp}).

\subsection{Parallel runtime for varying tensor dimension $d$}
We choose random $\mathcal{H}$-tensors where each tensor dimension is of size $n = 10\,000$ and all ranks are of size $k = 100$, i.e.
we choose the transfer tensors and the leaf frames as random tensors/matrices.
The choice of $n = 10\,000$ and $k = 100$ has initially been made to have the same amount 
of data on each
compute node (except for the root node): $n\cdot k = k^3$, cf. \eqref{equation:storageComplexity}.
However, as long as we choose the rank $k$ large enough, we observe the same behavior of the 
parallel runtime.
The experiments demonstrate that the parallel runtime of our algorithms scales as expected for the evaluation of tensor entries,
the inner product of two $\mathcal{H}$-tensors and for the truncation down to lower rank of an $\mathcal{H}$-tensor
including the orthogonalization.
(see Fig~\ref{figure:runtime1}).
\begin{figure}
\begin{subfigure}[b]{0.48\textwidth}
\input{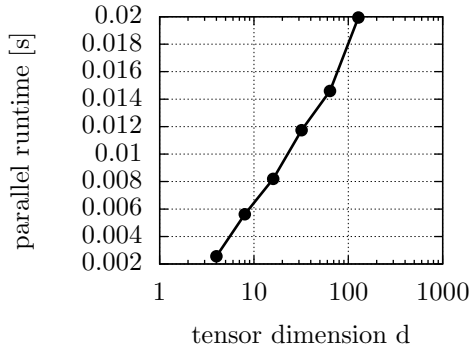}
\caption{Evaluation of one tensor entry.} 
\label{figure:plot-evaluation}
\end{subfigure}\quad
\begin{subfigure}[b]{0.48\textwidth}
\input{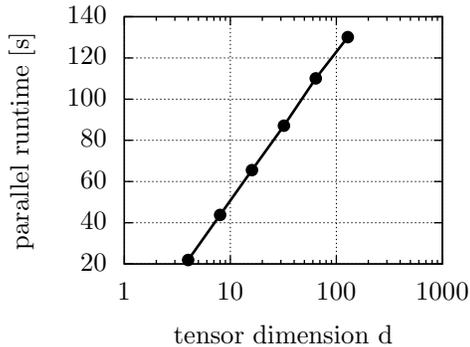}
\caption{Computation of the inner product.}
\label{figure:plot-dot}
\end{subfigure}\\
\begin{subfigure}[b]{0.48\textwidth}
\input{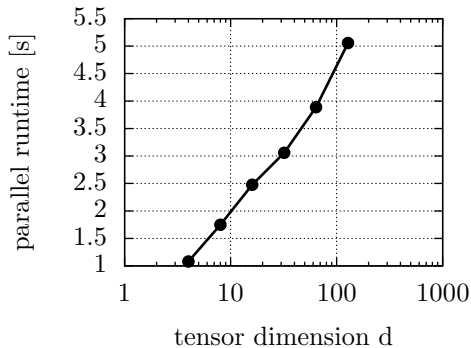}
\caption{Orthogonalization.}
\label{figure:plot-orth}
\end{subfigure}\quad
\begin{subfigure}[b]{0.48\textwidth}
\input{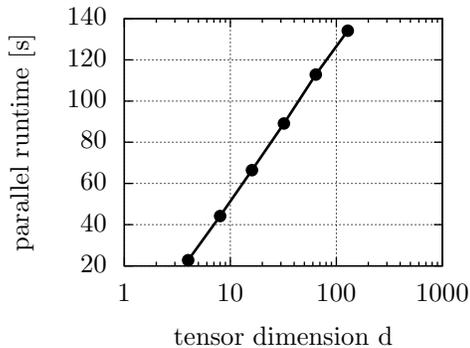}
\caption{Truncation of the orthogonalized $\mathcal{H}$-tensor.}
\label{figure:plot-truncation}
\end{subfigure}
\caption{Parallel runtimes for random $\mathcal{H}$-tensors. 
All tensor dimensions are of the same size $n = 10\,000$ and the tensor rank
on each node (except the root node) is $k = 100$. The parallel runtime of the algorithms for the above operations scales like
$\log(d)$.}
\label{figure:runtime1}
\end{figure}

From Fig.~\ref{figure:runtime2} we see that the parallel runtime 
for the application of an 
$\mathcal{H}$-operator to an $\mathcal{H}$-tensor is indeed rather independent of the tensor dimension $d$.
\begin{figure}
\input{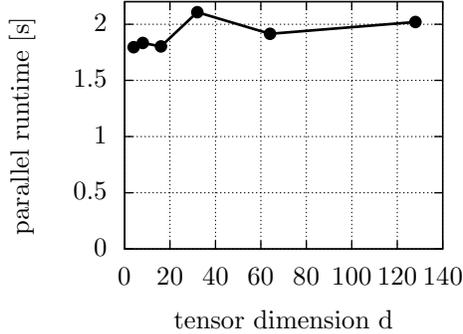}
\caption{The parallel runtime for the application of an $\mathcal{H}$-operator to an $\mathcal{H}$-tensor seems to be independent of the tensor dimension $d$ ($n = 400$, $k=20$).}
\label{figure:runtime2}
\end{figure}

\subsection{Parallel runtime for varying tensor rank $k$}
For varying tensor rank $k$ most of our algorithms reach the predicted complexity estimates 
(see Fig.~\ref{figure:runtime3}).
Only the orthogonalization of an $\mathcal{H}$-tensor behaves more like $\mathcal{O}(k^3)$ 
rather than the expected 
$\mathcal{O}(k^4)$. This might be due to optimizations in the LAPACK routines which we are using.

The ranks for the $\mathcal{H}$-operator are always chosen equal to those of the $\mathcal{H}$-tensor, which the operator is
applied to.
\begin{figure}
\begin{subfigure}[b]{0.48\textwidth}
\input{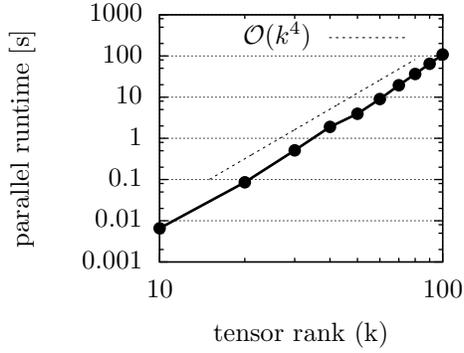}
\caption{Dot product for $n=10\,000$.} 
\label{figure:plot-dot-ranks}
\end{subfigure}\quad
\begin{subfigure}[b]{0.48\textwidth}
\input{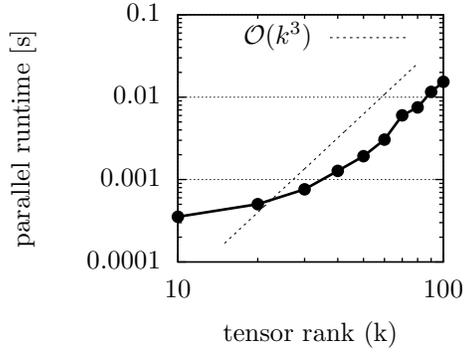}
\caption{Evaluation for $n = 10\,000$.} 
\label{figure:plot-evaluation-ranks}
\end{subfigure}\\
\begin{subfigure}[b]{0.48\textwidth}
\input{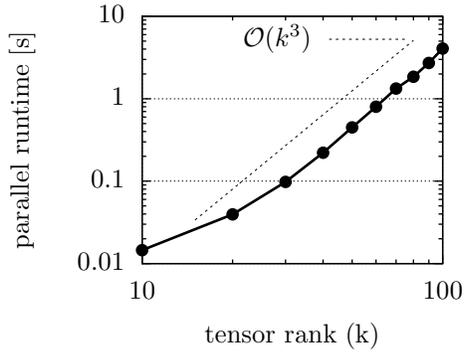}
\caption{Orthogonalization for $n=10\,000$.} 
\label{figure:plot-orth-ranks}
\end{subfigure}\quad
\begin{subfigure}[b]{0.48\textwidth}
\input{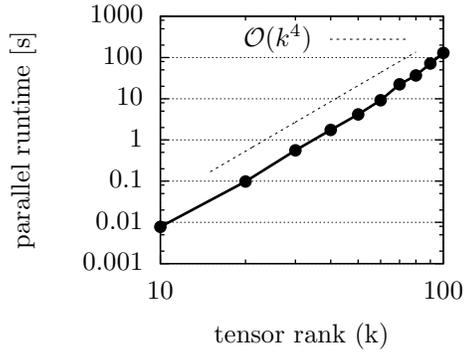}
\caption{Truncation for $n=10\,000$.} 
\label{figure:plot-truncation-ranks}
\end{subfigure}\\
\begin{subfigure}[b]{0.48\textwidth}
\input{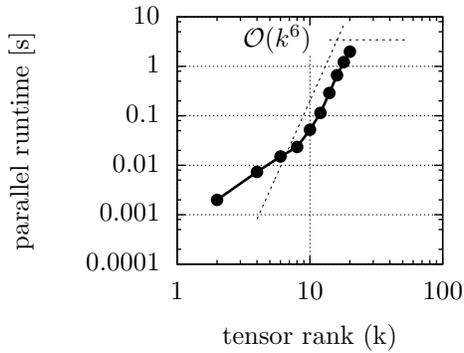}
\caption{Operator for $n=400$.} 
\label{figure:plot-operator-ranks}
\end{subfigure}
\caption{Parallel runtimes for tensor dimension $d = 64$, fixed sizes $n$ of the tensor dimensions and varying ranks $k$.}
\label{figure:runtime3}
\end{figure}

\section{Application to iterative methods}\label{Sec:multigrid}
Consider an $\mathcal{H}$-operator $\mathcal{A}\in\mathbb{R}^{\mathcal{I}\times\mathcal{I}}$
together with a suitable $\mathcal{H}$-tensor $B\in\mathbb{R}^{\mathcal{I}}$ as right-hand
side, where $\mathcal{I} = \mathcal{I}_1\times\cdots\times\mathcal{I}_d$ is a product index set
of dimension $d\in\mathbb{N}$. Then we can use the parallel algorithms for $\mathcal{H}$-tensor
arithmetic (cf. Section~\ref{section:arithmetic}) to solve $\mathcal{A}X = B$, 
$X\in\mathbb{R}^{\mathcal{I}}$, directly in the Hierarchical Tucker format.

In this section we apply a multigrid method to a diffusion equation
(\textit{"cookie problem"}, see Section~\ref{subsection:cookie}), which depends on 9 parameters.
A Finite Element discretizations leads to a parameter-dependent linear system which in our case
can directly
be transferred to a linear equation in the Hierarchical Tucker format.
All arithmetic operations are therefore carried out between tensors in 
the Hierarchical Tucker format.
As a smoother for the multigrid method we choose a Richardson method since it can easily be 
transferred to the Hierarchical Tucker format. On the coarsest grid level we use a CG algorithm
to solve the defect equation.
Each parameter will correspond to one tensor dimension. 
Together with the spacial component of the problem we
will thus end up with tensors of dimension $d=10$.

\subsection{The cookie problem}\label{subsection:cookie}
Inspired by \cite{KreTob11} we take the \textit{cookie problem}
\begin{equation}\label{equation:cookieProblem}
\left\{
\begin{array}{rrrr}
-\nabla (\sigma(x)\nabla u) & = & 1 & \mbox{in } \Omega \\
u & = & 0 & \mbox{on } \partial\Omega
\end{array}
\right.
\end{equation}
as a test problem, where our "cookies" are squares (instead of circles):

In the domain $\Omega = [0,7]\times [0,7]$ we define 9 rectangles ("cookies")
\begin{equation*}
\mathcal{D}_{\mu} = \{x\in\Omega\; :\; \|x - c_{\mu}\|_{\infty} < 0.5 \}\; ,\quad\mu = 1,\ldots ,9,
\end{equation*}
with  midpoints 
\begin{align*}
c_1 = (1.5,\,1.5)\;,\quad & c_2 = (3.5,\,1.5)\; ,\quad c_3 = (5.5,\,1.5)\; , \\
c_4 = (1.5,\,3.5)\;,\quad & c_5 = (3.5,\,3.5)\; ,\quad c_6 = (5.5,\,3.5)\; ,\\
c_7 = (1.5,\,5.5)\;,\quad & c_8 = (3.5,\,5.5)\; ,\quad c_9 = (5.5,\,5.5).
\end{align*}
The diffusion coefficient $\sigma(x)$, $x\in\Omega$, is defined as
\begin{equation}\label{equation:cookieDiffusion}
\sigma(x) = \left\{
\begin{array}{cll}
1+\alpha_{\mu} & \mbox{if} & x\in\mathcal{D}_{\mu}\; ,\quad \mu = 1,\ldots ,9, \\
 1 & \mbox{if} & x\notin\bigcup_{\mu = 1}^{9}\mathcal{D}_{\mu},
\end{array}
\right.
\end{equation}
i.e. depending on $9$ parameters $a_{\mu}$, $\mu = 1,\ldots ,9$ (cf. Fig.~\ref{figure:coarsegrid}).
\begin{figure}
\begin{center}
\includegraphics[scale=0.25]{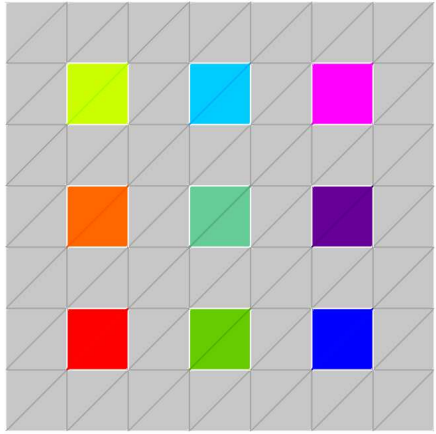}
\end{center}
\caption{Domain $\Omega = [0,7]\times [0,7]$ of the cookie problem. The cookies $\mathcal{D}_\mu$, $\mu = 1,\ldots ,9$,
 are
squares of side length $1$.
Outside the cookies (grey area) the diffusion coefficient equals one: $\sigma(x)\equiv 1$. 
On each cookie the diffusion coefficient is constant and controlled by a parameter: $\sigma(x) \equiv 1 + \alpha_\mu$, if
$x\in\mathcal{D}_\mu$, $\mu = 1,\ldots ,9$. \\Grid and image generated with ProMesh, \texttt{www.promesh3d.com}.}
\label{figure:coarsegrid}
\end{figure}
A Finite Element discretization of \eqref{equation:cookieProblem} with \eqref{equation:cookieDiffusion} 
leads to a parameter-dependent linear system
\begin{equation}\label{equation:cookieDiscrete}
A(\alpha_1,\alpha_2,\ldots ,\alpha_9) \cdot x(\alpha_1,\alpha_2,\ldots ,\alpha_9) = b,
\end{equation}
where the right-hand side $b$
is the same for each parameter combination $(\alpha_1,\alpha_2,\ldots ,\alpha_9)$.

Due to the affine structure of the operator in \eqref{equation:cookieProblem}
with \eqref{equation:cookieDiffusion}, the matrix in
\eqref{equation:cookieDiscrete} can directly be represented as an $\mathcal{H}$-operator:
Let $A_0\in\mathbb{R}^{n\times n}$ be the stiffness matrix of a Finite Element discretization of
\eqref{equation:cookieProblem} with $\sigma\equiv 1$ and let $A_\mu\in\mathbb{R}^{n\times n}$, 
$\mu =1,\ldots ,9$,
be the stiffness matrices which would result from a Finite Element discretization of
\eqref{equation:cookieProblem} on the same Finite Element space but with
$\sigma\equiv \mathbf{1}_{\mathcal{D}_\mu}$ instead of $\sigma\equiv 1$, where
$\mathbf{1}_{\mathcal{D}_\mu}\,:\, \Omega\to\{0,1\}$ is the indicator function of $\mathcal{D}_\mu$.
Since the diffusion coefficient \eqref{equation:cookieDiffusion} can be decomposed as
\begin{equation*}
\sigma \equiv 1 + \sum_{\mu = 1}^{9} \alpha_{\mu}\mathbf{1}_{\mathcal{D}_\mu},
\end{equation*}
the matrix $A(\alpha_1,\alpha_2,\ldots ,\alpha_9)$ can be decomposed as
\begin{equation}\label{equation:operatorStructure}
A(\alpha_1,\alpha_2,\ldots ,\alpha_9, i) = A_0(i) + \alpha_1 A_1(i) + \alpha_2 A_2(i) + \cdots + 
\alpha_9 A_9(i),\quad i = 1,\ldots ,n^2,
\end{equation}
where we combined the two matrix indices to one index $i$ 
($\mathbb{R}^{n\times n}\cong\mathbb{R}^{n^2}$). Assuming discretizations
\begin{equation*}
a_{\mu,i_{\mu}},\quad i_\mu = 1,\ldots ,n_\mu,\quad \mu = 1,\ldots ,d,
\end{equation*}
for the parameters, we can represent \eqref{equation:operatorStructure} in the Hierarchical Tucker
format, where we take the underlying tree $T_{10}$ as the balanced tree 
described in
Secion~\ref{section:distributedtensors}.
We choose the first $9$ leaf frames $U_{\{\mu\}}\in\mathbb{R}^{n_{\mu}\times 10}$,
$\mu = 1,\ldots ,9$, which correspond to the $9$ parameter directions, as
\begin{equation}\label{equation:parameterLeafFrames}
(U_{\{\mu\}})_{i_\mu,j} = 
\left\{
\begin{array}{ccl}
1 & \mbox{if} & j\neq \mu + 1\\
\alpha_{\mu,i_\mu} & \mbox{if} & j = \mu + 1
\end{array}
\right.\, ,\; i_\mu = 1,\ldots n_{\mu}\,,\; j = 1,\ldots ,10\,,\; \mu = 1,\ldots ,9.
\end{equation}
The last leaf frame $U_{\{10\}}\in\mathbb{R}^{n^2\times 10}$ 
just consists of the matrices $A_0,\ldots ,A_9$ 
(written as columns):
\begin{equation*}
(U_{10})_{i_{10},j} = A_{j-1}(i_{10})\; ,\quad i_{10} = 1,\ldots ,n^2\; ,\quad
j = 1,\ldots ,10.
\end{equation*}
The transfer tensors are diagonal tensors with ones on the diagonal.

One easily verifies that the above construction yields a representation of the
parameter dependent matrix $A$ as $\mathcal{H}$-operator. 
This representation is, however, not minimal: Each of the leaf frames in 
\eqref{equation:parameterLeafFrames} contains nine times the column $(1,\ldots ,1)^{\top}$. 
By only keeping this column once per leaf frame (and adjusting the transfer tensors accordingly),
we can always achieve Hierarchical ranks of $2$ for these leaf frames.

We thus get an equation $\mathcal{A}X = B$ in the Hierarchical Tucker format, where the $\mathcal{H}$-operator has
Hierarchical ranks $k_t = \rank(\mathcal{M}_t(\mathcal{A}))$, 
$t\in T_d$, of at most $d = 10$ and the right-hand side
$B$ is of rank $1$, since it only depends on the spacial variable and not on the parameters. The solution tensor $X$ then contains the solutions of \eqref{equation:cookieDiscrete} for all possible 
parameter combinations $(\alpha_1,\alpha_2,\ldots ,\alpha_9)$ and allows for further arithmetic operations in
the Hierarchical Tensor format, as e.g. computing averages over spacial domains or over parameters,
or computing the expected value or the variance with respect to a probability distribution of the parameters,
where the probability distribution has to be represented/approximated in the Hierarchical Tucker format
(which is very easy for the case of independent parameters).

\subsection{The CG method in parallel $\mathcal{H}$-tensor arithmetic}\label{subsection:CG}
We use a CG algorithm (Algorithm~\ref{algorithm:CG}) to solve the equation $\mathcal{A}X = B$ of 
the cookie problem 
(cf. Section~\ref{subsection:cookie}) 
on the coarsest grid, which is shown in Fig.~\ref{figure:coarsegrid}.
We let each parameter $a_{\mu}$, $\mu = 1,\ldots ,9$, vary between $0.5$ and $1.5$, using 10 equidistant
grid points in $[0.5,1.5]$. The full tensor $X$ on the coarsest grid would therefore consist of
\begin{equation*}
\#(\mbox{inner grid points}) \cdot 10^9 = 36\cdot 10^9\; \mbox{entries},
\end{equation*}
whereas in the Hierarchical Tucker format, the number of entries can be bounded by
\begin{align*}
& (\# \mbox{inner grid points} + 9\cdot 10)\cdot k_{\max} + \#(\mbox{inner tree nodes})\cdot k_{\max}^3 + k_{\max}^2 \\
= & (36+9\cdot 10)\cdot k_{\max} + 8\cdot k_{\max}^3 + k_{\max}^2 = 126\cdot k_{\max} + 8\cdot k_{\max}^3 + k_{\max}^2\;\mbox{entries},
\end{align*}
where $k_{\max}$ is an upper bound for the Hierarchical ranks.
For $k_{\max}=50$ these would be about $10^6$ entries.
\begin{algorithm}
\KwData{operator $\mathcal{A}$, right-hand side $B$ in the Hierarchical Tucker format}
\KwResult{Approximate solution $X$ in the Hierarchical Tucker format}
$X_0 = B$\;
$R_0 = B - \mathcal{A}X_0$\;
$R_0 = \mathcal{T}(R_0)$\;
$D_0 = R_0$\;
$j = 0$\;
\While{ $\sqrt{\langle R_j, R_j\rangle } >$ EPS {\bf and} $j < $ MAX\_CG\_STEPS }
{
    $Z = \mathcal{A}D_j$\;
    $Z = \mathcal{T}(Z)$\;
    $\alpha_j = \langle R_j,R_j\rangle / \langle D_j, Z\rangle$\;
    $X_{j+1} = X_{j} + \alpha_{j} D_j$\;
    $X_{j+1} = \mathcal{T}(X_{j+1})$\;
    $R_{j+1} = R_{j} - \alpha_{j} Z$\;
    $R_{j+1} = \mathcal{T}(R_{j+1})$\;
    $\beta_{j} = \langle R_{j+1}, R_{j+1}\rangle / \langle R_{j}, R_{j}\rangle$\;
    $D_{j+1} = \beta_{j}D_{j} + R_{j+1}$\;
    $D_{j+1} = \mathcal{T}(D_{j+1})$\;
    $j = j+1$\;
}
$X = X_j$\;
\caption{CG algorithm to solve $\mathcal{A}X = B$ in the Hierarchical Tucker format including truncation.
The truncation $\mathcal{T}$ can be carried out either down to fixed rank $k$ ($\mathcal{T} = \mathcal{T}_k$) 
or with prescribed accuracy $\varepsilon$ ($\mathcal{T}= \mathcal{T}_{\varepsilon}$).
The stopping criterion depends on a lower bound EPS for the norm of the residual and on an upper bound 
MAX\_CG\_STEPS for the number of CG steps.}

\label{algorithm:CG}
\end{algorithm}

In the CG algorithm (Algorithm~\ref{algorithm:CG}) we use our parallel algorithms for arithmetic operations between $\mathcal{H}$-tensors (cf. Section~\ref{section:arithmetic}). 
After each addition of two $\mathcal{H}$-tensors and after each 
application of the $\mathcal{H}$-operator $\mathcal{A}$ to an $\mathcal{H}$-tensor we truncate the result
 down to lower rank, where we prescribe some accuracy $\varepsilon$
such that the truncation error is smaller than $\varepsilon$. 

The following Lemma~\ref{lemma:perturbedCG} considers the simplified case, where after each 
CG step in
Algorithm~\ref{algorithm:CG} ($j=0,1,2,\ldots$) the iterate $X_j$ is truncated, i.e. 
truncations in between are neglected. For this simplified case we can reach any accuracy 
$\delta >0$ of the approximate solution (i.e. $\|X - X_j\|<\delta$ for some $j$) 
if we choose the upper bound $\varepsilon > 0$ for the 
truncation error small enough.
Similar results have already been published in \cite{Matthies2012} for the general case of
convergent iterative processes in Banach spaces.
\begin{lemma}[Convergence of the CG algorithm including truncation]\label{lemma:perturbedCG}
Consider the linear equation $\mathcal{A}X=B$ with symmetric and positive-definite $\mathcal{A}$.
Let $\mathcal{C}$ denote the operation of performing one exact CG step according to the linear
system $\mathcal{A}X =B$, i.e. when $X_j$ is the result of $j$ exact CG steps, 
$X_{j+1} = \mathcal{C}X_j$
would be the next iterate in the exact CG method. 
Furthermore let $\mathcal{T}_{\varepsilon}$ be the 
truncation down to lower ranks with error smaller than $\varepsilon$.

For arbitrary $\delta > 0$ and arbitrary starting tensor $X_0$ we can achieve
\begin{equation*}
\|X - (\mathcal{T}_{\varepsilon}\mathcal{C})^j X_0\|_{\mathcal{A}} < \delta
\end{equation*}
for some $j\in\mathbb{N}$, if we choose
\begin{equation*}
\varepsilon < \delta / C
\end{equation*}
with $C:= \sqrt{\|\mathcal{A}\|} (\cond(\mathcal{A}) + 1)$, where 
$\|\mathcal{A}\| = \lambda_{\max}$ is the spectral norm (largest eigenvalue) of $\mathcal{A}$, 
$\cond(\mathcal{A}) = \lambda_{\max}/\lambda_{\min}$ is the condition number 
(ratio between largest and smallest eigenvalue) of $\mathcal{A}$ and
$\|\cdot \|_{\mathcal{A}}$ is the $\mathcal{A}$-norm, defined as 
$\|X\|_{\mathcal{A}} :=\sqrt{\langle X,\mathcal{A}X\rangle }$.
\end{lemma}
\begin{proof}
For the exact CG algorithm one can show (cf. \cite{Meurant06})
\begin{equation}\label{equation:exactCG}
\|X-\mathcal{C}^j X_0\|_{\mathcal{A}} \leq \gamma \|X - \mathcal{C}^{j-1}X_0\|_{\mathcal{A}}
\end{equation}
with $\gamma = (\cond(\mathcal{A}) -1)/(\cond(\mathcal{A}) + 1)$.

We now split the error of one perturbed CG step into
the error of one exact CG step and one truncation error:
\begin{equation*}
\|X-(\mathcal{T}_{\varepsilon}\mathcal{C}) X_0\|_{\mathcal{A}}
\leq \| X - \mathcal{C} X_0\|_{\mathcal{A}} + \| \mathcal{C} X_0 - \mathcal{T}_{\varepsilon}\mathcal{C}X_0\|_{\mathcal{A}}.
\end{equation*}
The first error can be estimated by \eqref{equation:exactCG}. 
The second error is first transformed to the standard norm
$\|\cdot \|$ (Euclidian norm/Frobenius norm), which we use to measure the truncation error, since it can then
be bounded by $\varepsilon$:
\begin{align*}
\|X-(\mathcal{T}_{\varepsilon}\mathcal{C}) X_0\|_{\mathcal{A}} &
\leq \gamma\|X-X_0\|_{\mathcal{A}} + \sqrt{\lambda_{\max}}\|\mathcal{C}X_0 - \mathcal{T}_{\varepsilon}\mathcal{C}X_0\| \\
 & \leq \gamma\| X-X_0\|_{\mathcal{A}} + \varepsilon\sqrt{\lambda_{\max}} .
\end{align*}
Similarly (by induction) we can get a bound for the error of $j$ perturbed 
CG steps:
\begin{align*}
\|X - (\mathcal{T}_{\varepsilon}\mathcal{C})^j X_0\|_{\mathcal{A}} &
\leq \gamma \| X-(\mathcal{T}_{\varepsilon}\mathcal{C})^{j-1} X_0\|_{\mathcal{A}} 
+ \varepsilon\sqrt{\lambda_{\max}} \\
 & \leq \gamma( \gamma\|X-(\mathcal{T}_{\varepsilon}\mathcal{C})^{j-2} X_0\|_{\mathcal{A}} + \varepsilon\sqrt{\lambda_{\max}} ) + \varepsilon\sqrt{\lambda_{\max}} \\
& = \gamma ^2\|X - (\mathcal{T}_{\varepsilon}\mathcal{C})^{j-2}X_0\|_{\mathcal{A}} + \varepsilon\sqrt{\lambda_{\max}}(\gamma + 1 ) \\
& \leq \cdots\leq \gamma^{j}\|X - X_0\|_{\mathcal{A}} + \varepsilon\sqrt{\lambda_{\max}}\left(\sum_{\ell=0}^{j-1}\gamma^\ell\right).
\end{align*}
The sum can be bounded by $(1-\gamma)^{-1}$, thus we have
\begin{equation*}
\|X - (\mathcal{T}_{\varepsilon}\mathcal{C})^j X_0\|_{\mathcal{A}} \leq
\gamma^j \|X - X_0\|_{\mathcal{A}} + \varepsilon\sqrt{\lambda_{\max}} (1-\gamma)^{-1}.
\end{equation*}
Since $\gamma^j \to 0$ for $j\to\infty$, we can find $j\in\mathbb{N}$ such that $\gamma^j\|X-X_0\|_{\mathcal{A}} < \delta / 2$. To achieve $\|X-(\mathcal{T}_{\varepsilon}\mathcal{C})^j X_0\|_{\mathcal{A}} <\delta$ we thus
have to choose $\varepsilon$ in such a way that
\begin{equation*}
\varepsilon\sqrt{\lambda_{\max}}(1-\gamma)^{-1} < \frac{\delta}{2}\quad\Leftrightarrow\quad
\varepsilon < \frac{\delta}{C}\quad\mbox{with}\quad C = \sqrt{\lambda_{\max}}(\cond(\mathcal{A})+1),
\end{equation*}
which completes the proof.
\end{proof}
\begin{remark}
\begin{enumerate}
\item Lemma~\ref{lemma:perturbedCG} only states the convergence of the perturbed CG algorithm, 
i.e. that in
principal we can reach arbitrary accuracy (absolute error of the computed solution) if the absolute 
truncation error is bounded by a sufficiently small $\varepsilon$. Lemma~\ref{lemma:perturbedCG} does not
reveal the convergence rate of such an algorithm.
The estimate 
\begin{equation}\label{equation:convergenceRate1}
\|X-\mathcal{C}^j X_0\|_{\mathcal{A}}\leq \left(\frac{\cond(A)-1}{\cond(A)+1}\right)^j
\|X-X_0\|_{\mathcal{A}},
\end{equation}
which we used in the proof,
is quite pessimistic. This bound can already be achieved for the method of steepest descent, where the residual $R_{j+1}$ of one step is typically only orthogonal to the search direction $D_j$ of the last step 
(cf. Algorithm~\ref{algorithm:CG} for notation). In the CG method, however, this residual is orthogonal to the search directions of all previous steps, which yields the (actually much better) estimate
\begin{equation}\label{equation:convergenceRate2}
\|X-\mathcal{C}^j X_0\|_{\mathcal{A}}\leq 2\left(\frac{\sqrt{\cond(A)}-1}{\sqrt{\cond(A)}+1}\right)^j
\|X-X_0\|_{\mathcal{A}}.
\end{equation}
In the proof of Lemma~\ref{lemma:perturbedCG} we used \eqref{equation:convergenceRate1} 
instead of \eqref{equation:convergenceRate2}, since
we were interested in the decay of the error for one single CG step (i.e. $j=1$), 
where \eqref{equation:convergenceRate2} can not be used for.

Moreover, in practice, the Euclidian norm of the error instead of the $\mathcal{A}$-norm will be of interest, which would
give us an additional factor $1/\sqrt{\lambda_{\min}} = \sqrt{\|\mathcal{A}^{-1}\|}$ for the constant $C$.

Much more precise analysis of perturbed CG methods can be found in \cite{Meurant06} or \cite{Greenbaum89}, where the perturbed
CG algorithm is reinterpreted as an exact CG algorithm for a perturbed operator.
\item Lemma~\ref{lemma:perturbedCG} assumes only one truncation of the iterate $X_j$ in each CG step.
In practice we truncate after each addition of two $\mathcal{H}$-tensors as well as after each application of
$\mathcal{A}$ to an $\mathcal{H}$-tensor (cf. Algorithm~\ref{algorithm:CG}). This includes also truncations 
of the residuals $R_j$ and the search directions $D_j$, which themselves 
are afterwards
involved in a quotient of scalar products. A precise analysis would therefore have to deal with the respective propagation of
the errors inside the algorithm, cf. \cite{Greenbaum89}, \cite{Meurant06}.
\item For our example ($\mathcal{H}$-operator $\mathcal{A}$ for the coarsest grid of the cookie problem 
with $10$ equidistant
points in $[0.5,1.5]$ per parameter) we can estimate $\lambda_{\max}\approx 11$ and $\lambda_{\min}\approx 0.4$,
which would yield $C \approx 95$ in Lemma~\ref{lemma:perturbedCG}, and $C \approx 150$ 
if we measure the error in the Euclidian norm instead of the $\mathcal{A}$-norm.
\end{enumerate}
\end{remark}
In our numerical experiments we used $\varepsilon = 10^{-4}$ as upper bound for the absolute error of each
truncation. Since the Hierarchical ranks can temporarily get rather large during the CG
algorithm (cf. \cite{BaGr13}), we prescribed 
additional bounds for the maximal rank,
by which we implicitly choose a larger $\varepsilon$. 
In Fig.~\ref{figure:cg-tests} we plotted the relative residual norm $\|\mathcal{A}X_j -B\| / \|B\|$
within 25 CG steps with different bounds for
the maximal rank. One observes the principle of Lemma~\ref{lemma:perturbedCG}: the higher the rank bound 
(i.e. the smaller we choose $\varepsilon$), the smaller the relative resiudal norm, which can be reached in the
CG algorithm.

The plotted residual norms are not the norms of the "residuals" $R_j$ computed 
in the perturbed CG algorithm, since these can deviate
from the exact residual $AX_j - B$ (cf. \cite{Greenbaum89}). 
We therefore computed the residual norm after each CG step, i.e.
computed $\sqrt{\langle \mathcal{A}X_j - B, \mathcal{A}X_j -B\rangle}$ in $\mathcal{H}$-tensor arithmetic.
\begin{figure}
\begin{subfigure}[b]{0.48\textwidth}
\input{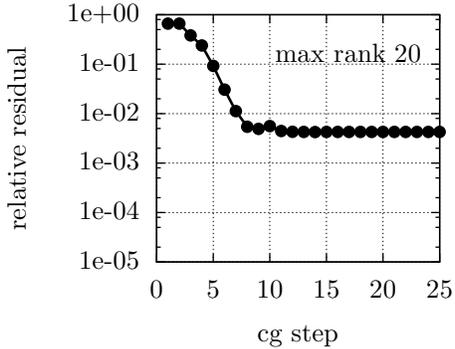}
\caption{Maximal rank $k=20$.} 
\label{figure:cg-tests-rank20}
\end{subfigure}\quad
\begin{subfigure}[b]{0.48\textwidth}
\input{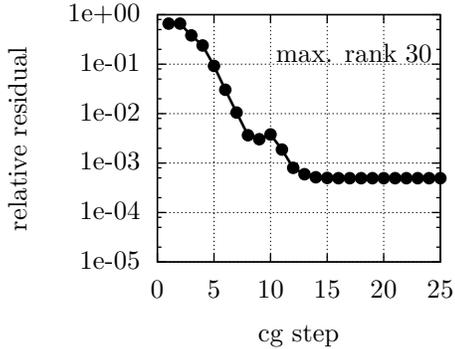}
\caption{Maximal rank $k=30$.}
\label{figure:cg-tests-rank30}
\end{subfigure}\\
\begin{subfigure}[b]{0.48\textwidth}
\input{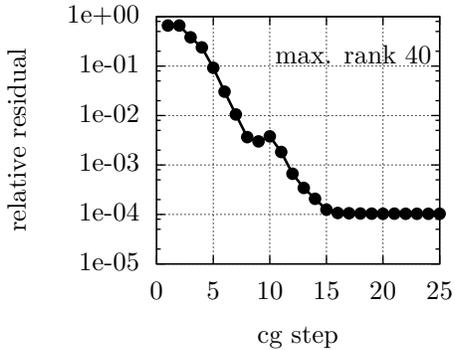}
\caption{Maximal rank $k=40$.}
\label{figure:cg-tests-rank40}
\end{subfigure}\quad
\begin{subfigure}[b]{0.48\textwidth}
\input{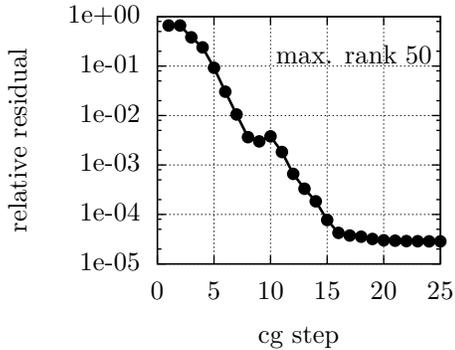}
\caption{Maximal rank $k=50$.}
\label{figure:cg-tests-rank50}
\end{subfigure}
\caption{CG algorithm using $\mathcal{H}$-arithmetic for the cookie problem on the coarsest grid, which is
depicted in Fig.~\ref{figure:coarsegrid}. For each of the nine parameters we use 10 equidistant grid points in 
$[\frac{1}{2},\frac{3}{2}]$. After each addition of $\mathcal{H}$-tensors and after each application of the
$\mathcal{H}$-operator $\mathcal{A}$ to an $\mathcal{H}$-tensor we truncate the resulting tensor down to
lower ranks with an absolute error less than $10^{-4}$. In addition we prescribe a maximal rank $k$, since the 
involved computations have complexity up to $\mathcal{O}(k^4)$.}
\label{figure:cg-tests}
\end{figure}

When applying our CG algorithms with fixed rank bounds on finer grids, 
we need more iterations until we achieve the maximal atainable accuracy with respect to
the rank bound, as one would
expect it due to the growing condition number of the operator $\mathcal{A}$.
In order to overcome this, one could use preconditioning.
We instead apply a multigrid algorithm with respect to the spacial component of the
problem (semi-coarsening) to solve $\mathcal{A}X=B$ on finer grids, which is described
in Section~\ref{subsection:multigridMethod}. 

Multigrid methods in the Hierarchical Tucker format have already been studied in
\cite{BaGr13}, where the number of iterations needed to solve the Poisson equation up to some
relative residual turned out to be independent of the tensor dimension $d$ 
(i.e. the number of parameters). This would make them usable for problems depending on a huge
number of parameters.

\subsection{A multigrid method in parallel $\mathcal{H}$-tensor arithmetic}\label{subsection:multigridMethod}
Multigrid methods have already been applied for low-rank matrices and 
$\mathcal{H}$-matrices in
\cite{GrHa2007} and \cite{Gr2008} in order to solve large scale Sylvester and Riccati
equations. In \cite{BaGr13} multigrid methods have been applied in $\mathcal{H}$-tensor
arithmetic to solve equations on high dimensional domains.

In order to compute a solution of higher resolution, we refine the coarse grid of 
Fig.~\ref{figure:coarsegrid} three
times. With every refinement step we add the midpoints of all edges as new grid points. By this we get the grid
hierarchy of Fig.~\ref{figure:gridHierarchy}, where the finest grid consist of $3025$ inner grid points.
\begin{figure}
\begin{center}
\includegraphics[scale=0.35]{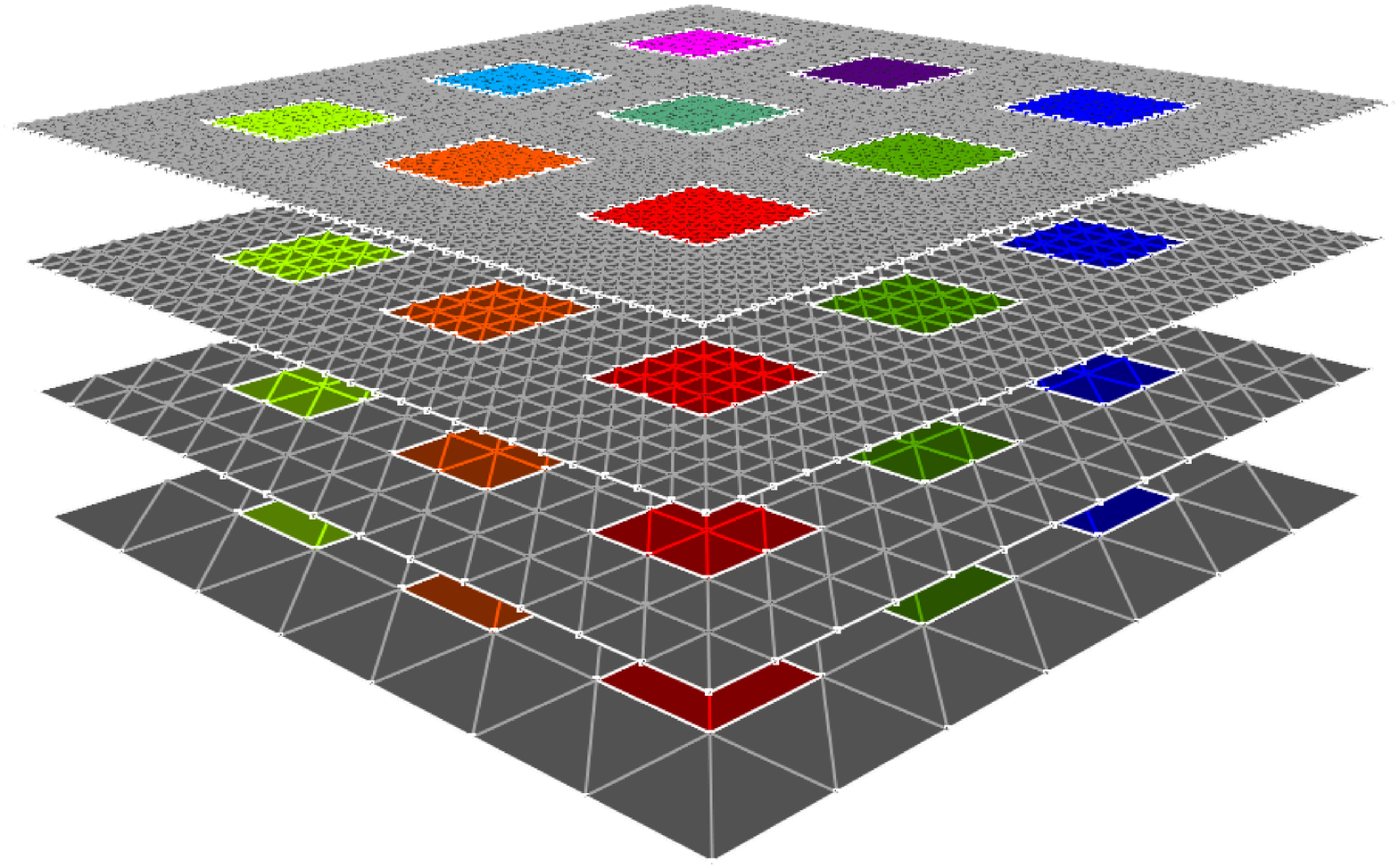}
\end{center}
\caption{Grid hierarchy for the multigrid method with $36$ inner grid points on the coarsest grid and $3025$
inner grid points on the finest grid.\\Grid and image generated with ProMesh, \texttt{www.promesh3d.com}.}
\label{figure:gridHierarchy}
\end{figure}

One multigrid step
(Algorithm~\ref{algorithm:multigrid}) starts on the finest grid and then recursively proceeds to the next 
coarser grid to solve the system $\mathcal{A}E = D$, where $D$ is the restriction of the residual 
$B - \mathcal{A}X$ to the next coarser grid. The solution $E$ is then prolongated back to the finer grid and
added to the current iterate to get the next iterate. For the multigrid 
method to work well we need 
to be able to compute good approximations of the error
on the next coarser grid. These can be obtained by several steps of a Richardson
method (Algorithm~\ref{algorithm:richardson}), which serves as a smoother for the error.
On the coarsest level we use the CG algorithm (cf. Section~\ref{subsection:CG}) to solve the system.
An introduction to multigrid algorithms together with convergence analysis can e.g. be found in \cite{Ha85}.
\begin{algorithm}[p]
\KwData{iterate $X_j$}
\KwResult{new iterate $X_{j+1}$}
$Z = \mathcal{A}X_j$\;
$Z = \mathcal{T}(Z)$\;
$Z = Z - B$\;
$Z = \mathcal{T}(Z)$\;
$X_{j+1} = X_j -\omega Z$\;
$X_{j+1} = \mathcal{T}(X_{j+1})$\;
\caption{$X_{j+1}=\;$\texttt{richardson\_step( $\mathcal{A}$, $X_{j}$, $B$ )} performs one step of the 
Richardson method
for $\mathcal{A}X = B$ in $\mathcal{H}$-tensor arithmetic, including
truncation. We use this method as a smoother for the multigrid method (Algorithm~\ref{algorithm:multigrid}). 
The truncation $\mathcal{T}$ can be carried out either down to fixed rank $k$ 
($\mathcal{T}=\mathcal{T}_k$) or with prescribed accuracy $\varepsilon$ 
($\mathcal{T} = \mathcal{T}_{\varepsilon}$). The optimal choice for $\omega$ is
$\omega_{opt}=2/(\lambda_{\min} + \lambda_{\max})$, where $\lambda_{\min}$ and $\lambda_{\max}$ are the minimal
and the maximal eigenvalues of the operator $\mathcal{A}$.}
\label{algorithm:richardson}
\end{algorithm}

\begin{algorithm}[p]
\KwData{iterate $X_j$}
\KwResult{new iterate $X_{j+1}$}
\eIf {on coarsest grid}
{
    Solve $\mathcal{A}X_{j+1}=B$ with CG (Algorithm~\ref{algorithm:CG})\;
}
{
    $5\times$ presmoothing: $X_j=\;$\texttt{richardson\_step( $\mathcal{A}$, $X_j$, $B$ )} 
    (Algorithm~\ref{algorithm:richardson})\;
    $R = B - \mathcal{A}X_j$\;
    $R = \mathcal{T}(R)$\;
    $D = \mathrm{restriction}(R)$\;
    $E = D$\;
    $E=\;$\texttt{multigrid\_step( $\mathcal{A}$, $E$, $D$ )}\;
    $P = \mathrm{prolongation}( E )$\;
    $X_{j+1} = X_j + P$\;
    $X_{j+1} = \mathcal{T}(X_{j+1})$\;
    $5\times$ postsmoothing: $X_{j+1}=\;$\texttt{richardson\_step( $\mathcal{A}$, $X_{j+1}$, $B$ )}
    (Algorithm~\ref{algorithm:richardson})\;
}
\caption{One multigrid step, $X_{j+1}=\;$\texttt{multigrid\_step( $\mathcal{A}$, $X_j$, $B$ )}, 
to solve ${\mathcal{A}X = B}$ in the
Hierarchical Tucker format, including truncation.
The truncation $\mathcal{T}$ can be carried out either down to fixed rank $k$ ($\mathcal{T}=\mathcal{T}_k$)
or with prescribed accuracy $\varepsilon$ ($\mathcal{T}=\mathcal{T}_{\varepsilon}$).}
\label{algorithm:multigrid}
\end{algorithm}

For linear methods (as the Richardson method and the multigrid method), one can obtain similar results to
Lemma~\ref{lemma:perturbedCG}: the best attainable accuracy of the solution depends on the accuracy of the 
truncations (see also \cite{Matthies2012}).

We performed $10$ multigrid steps for the equation $\mathcal{A}X=B$ on the finest grid
($3025$ inner grid points, cf. Fig.~\ref{figure:gridHierarchy}) with different upper bounds
for the Hierarchical ranks. In Fig.~\ref{figure:mg-tests} we show the relative residual norms
for upper rank bounds of $30$ and $50$. 
One observes the linear convergence of the multigrid method down to some lower threshold of the
relative residual norm.
This threshold decreases with increasing rank bounds, which is shown in Fig.~\ref{figure:mg-accuracy-rank}.
\begin{figure}
\begin{subfigure}[b]{0.48\textwidth}
\input{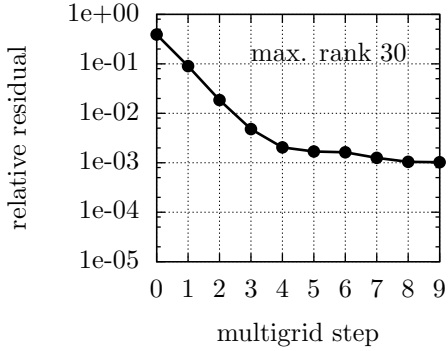}
\caption{Maximal rank $k=30$.} 
\label{figure:mg-tests-rank30}
\end{subfigure}\quad
\begin{subfigure}[b]{0.48\textwidth}
\input{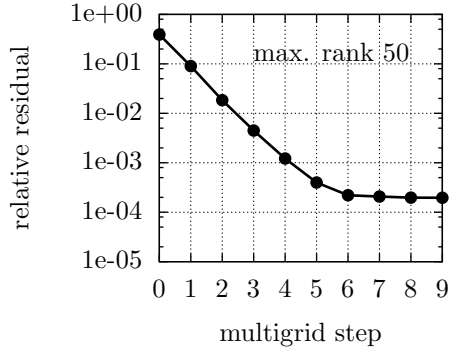}
\caption{Maximal rank $k=50$.}
\label{figure:mg-tests-rank50}
\end{subfigure}
\caption{Multigrid method using $\mathcal{H}$-arithmetic for the cookie problem on the grid hierarchy from
Fig.~\ref{figure:gridHierarchy}, once with maximal rank $k=30$ and once with maximal rank $k=50$. 
On the coarsest grid we apply the CG solver from Section~\ref{subsection:CG}.}
\label{figure:mg-tests}
\end{figure}
\begin{figure}
\begin{center}
\input{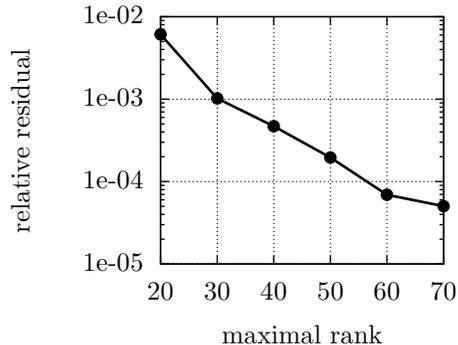}
\end{center}
\caption{Dependence of the relative residual norm on the upper rank bound 
in the multigrid method.
For each rank bound $20,30,\ldots ,70$ we plotted the best accuracy we obtained after 10
multigrid steps.}
\label{figure:mg-accuracy-rank}
\end{figure}

\section{Conclusions}\label{sec:conclusions}
In this article we presented parallel algorithms which we developed for basic arithmetic operations between distributed
tensors in the Hierarchical Tucker format. Throughout the article we assumed the tensor data to
be distributed according to the underlying tree of the Hierarchical
Tucker format: The data for each tree node is stored on its own compute node.
Since our algorithms typically need communication between different tree levels, we minimized the
number of tree levels by choosing a balanced tree. For a tensor of dimension $d$ the number of
tree levels grows like $\log(d)$, i.e. a level-wise parallelization of our 
algorithms yields a parallel runtime, which grows like $\log(d)$.
We validated the expected $\log(d)$-growth of the parallel runtime by several runtime tests in
Section~\ref{section:runtime}. 
The $\log(d)$-growth of the parallel runtime makes our algorithms applicable for high-dimensional
problems which will be part of future work.

Despite the $\log(d)$-growth of the parallel runtime, we still have a complexity of 
at least $\mathcal{O}(k^3)$ for each compute node 
(cf. Section~\ref{section:runtime}), where $k\in\mathbb{N}$ is the Hierarchical rank
for the respective node.
We would therefore expect a further reduction of the runtime by using shared
memory parallelization on each node (e.g. \texttt{OpenMP}).

In this article we used a parallelized version of the Hierarchical SVD from \cite{Gr10} to
truncate tensors in the Hierarchical Tucker format.
Since this algorithm involves the computation of $B^{\top}B$ and $BB^{\top}$ for 
matrices $B$, it may lead to a loss of accuracy. 
It will be interesting in future research to transfer the techniques from \cite{Etter16,StWhite13}
to the Hierarchical SVD.

One interesting application is parameter-dependent problems, which we illustrated 
by a model problem in 
Section~\ref{Sec:multigrid}. We applied a multigrid method including our parallel algorithms 
for $\mathcal{H}$-arithmetic to
a diffusion equation which depends on $9$ parameters. 
Our experiments indicate that the relative residual norm in the iterative method 
is bounded from below by
some barrier which depends on the accuracy of the $\mathcal{H}$-tensor truncations during the 
method.

For larger problems preconditioners should be applied,
when using the CG method (cf. \cite{KreTob11}), which we left out here.
Moreover it could be advantageous to apply multilevel sampling techniques, which were introduced
in \cite{Ba17}.

Our parallel algorithms can as well be used for the postprocessing of (distributed) tensors 
in the Hierarchical Tucker format, which might have been assembled by some other method 
(e.g. a sampling method).
This could be of interest in the field of uncertainty quantification and could be used for parameter fitting,
i.e. adapting model parameters based on some quantity of interest of the solution, which is either known or
can be estimated.

\bibliographystyle{siam}
\bibliography{tensor}

\end{document}